\documentclass[reqno,12pt]{amsart}

\usepackage{amscd,amssymb}

\numberwithin{equation}{section}

\newcommand{\1}{\mathbf {1}}
\newcommand{\N}{\mathbb{N}}
\newcommand{\Z}{\mathbb{Z}}
\newcommand{\R}{\mathbb{R}}
\newcommand{\C}{\mathbb{C}}
\newcommand{\h}{\mathfrak{h}}

\newcommand{\CI}{\mathcal{I}}
\newcommand{\CA}{\mathcal{A}}

\newcommand{\CF}{\mathcal{F}}
\newcommand{\CJ}{\mathcal{J}}

\newcommand{\CL}{\mathcal{L}}
\newcommand{\CN}{\mathcal{N}}

\newcommand{\CU}{\mathcal{U}}

\newcommand{\CW}{\mathcal{W}}

\newcommand{\mraff}{\mathrm{aff}}
\newcommand{\mrprin}{\mathrm{prin}}

\newcommand{\la}{\langle}
\newcommand{\ra}{\rangle}
\newcommand{\bu}{\mathbf{u}}
\newcommand{\bv}{\mathbf{v}}

\DeclareMathOperator{\Aut}{Aut}

\DeclareMathOperator{\Hom}{Hom}
\DeclareMathOperator{\gr}{gr}
\DeclareMathOperator{\Ind}{Ind}
\DeclareMathOperator{\Ker}{Ker}
\DeclareMathOperator{\Res}{Res}
\DeclareMathOperator{\spn}{span}

\DeclareMathOperator{\wt}{wt}

\newtheorem{thm}{Theorem}[section]
\newtheorem{prop}[thm]{Proposition}
\newtheorem{lem}[thm]{Lemma}
\newtheorem{cor}[thm]{Corollary}

\begin{document}

\title[parafermion VOAs]
{Zhu's algebra, $C_2$-algebra and $C_2$-cofiniteness of parafermion vertex operator algebras}

\author[T. Arakawa]{Tomoyuki Arakawa}
\address{Research Institute for Mathematical Sciences,
Kyoto University, Kyoto 606-8502, Japan}
\email{arakawa@kurims.kyoto-u.ac.jp}

\author[C. H. Lam]{Ching Hung Lam}
\address{ Institute of Mathematics, Academia Sinica, Taipei 10617, Taiwan, R.O.C.}
\email{chlam@math.sinica.edu.tw}

\author[H. Yamada]{Hiromichi Yamada}
\address{Department of Mathematics, Hitotsubashi University, Kunitachi,
Tokyo 186-8601, Japan}
\email{yamada@econ.hit-u.ac.jp}



\begin{abstract}
We study Zhu's algebra, $C_2$-algebra and $C_2$-cofiniteness 
of parafermion vertex operator algebras. 
We first give a detailed study of Zhu's algebra and $C_2$-algebra of parafermion vertex 
operator algebras associated with the affine Kac-Moody Lie algebra $\widehat{sl}_2$.
We show that they have the same dimension and Zhu's algebra is semisimple.
The classification of irreducible modules is also established.
Finally, we prove that the parafermion vertex operator algebras for any 
finite dimensional simple Lie algebras are $C_2$-cofinite.
\end{abstract}

\maketitle

\section{Introduction}

Let $\mathfrak{g}$ be a finite dimensional simple Lie algebra and 
$\widehat{\mathfrak{g}}$ the affine Kac-Moody Lie algebra
associated with $\mathfrak{g}$. Let $V_{\widehat{\mathfrak{g}}}(k,0)$ be the vacuum Weyl module
for $\widehat{\mathfrak{g}}$ with level $k$, where $k$ is a
positive integer and $L_{\widehat{\mathfrak{g}}}(k,0)$ its
simple quotient. The vertex operator algebra
$L_{\widehat{\mathfrak{g}}}(k,0)$ contains a Heisenberg vertex
operator algebra corresponding to a Cartan subalgebra $\mathfrak{h}$
of $\mathfrak{g}$. The commutant $K(\mathfrak{g},k)$ of the
Heisenberg vertex operator algebra in
$L_{\widehat{\mathfrak{g}}}(k,0)$ is called a parafermion vertex
operator algebra.

Recently, the structure of the parafermion vertex operator algebra $K(\mathfrak{g},k)$ 
was studied by C. Dong and Q. Wang \cite{DW, DW2}. 
In particular, 
they noticed the importance of the special case $\mathfrak{g} = sl_2$.  
The parafermion vertex operator algebra $K(sl_2, k)$ is known to be a $W$-algebra. 
It was also shown in \cite{DLWY, DLY} that $K(sl_2, k)$ is isomorphic to 
the simple quotient of the $W$-algebra $W(2,3,4,5)$ of \cite{BEHHH, Hornfeck}. 
In fact, it is expected that  the parafermion vertex operator algebra 
$K(sl_2, k)$ is isomorphic to a $W$-algebra $W_\ell(sl_k,f_\mrprin)$ of 
\cite{Arakawa, FKW} with $\ell = \frac{k+2}{k+1} - k$ and $f_\mrprin$ 
a principal nilpotent element of $sl_k$.  

For a vertex operator algebra $V$, Zhu \cite{Zhu} introduced two
intrinsic associative algebras, one is Zhu's algebra $A(V)$ and
the other is Zhu's $C_2$-algebra $V/C_2(V)$. We denote $V/C_2(V)$ by $R_V$
for simplicity of notation. 
In this article, we study Zhu's algebra, Zhu's $C_2$-algebra and $C_2$-cofiniteness 
of parafermion vertex operator algebras.  There are two main parts.
The first one is a detailed analysis of Zhu's algebra and $C_2$-algebra
of parafermion vertex operator algebras $K(sl_2,k)$ associated with
$\widehat{sl}_2$.  
The results of this paper concerning $\CW = K(sl_2,k)$ can be summarized as follows 
(Theorems \ref{thm:embedding_RW_in_RL}, \ref{thm:dim_Zhu_CW}, 
\ref{thm:Zhu_alg_CW} and \ref{thm:projectivity-CW}).
\begin{enumerate}
\item Zhu's $C_2$-algebra $R_\CW$ of $\CW$ is of dimension $k(k+1)/2$.
\item Zhu's algebra $A(\CW)$ of $\CW$ is semisimple and of dimension $k(k+1)/2$.
\item $M^{i,j}$, $0 \le i \le k$, $0 \le j \le i-1$ constructed in \cite{DLY}
form a complete set of isomorphism classes of irreducible $\CW$-modules.
\item $\CW$ is projective as a $\CW$-module.
\end{enumerate}
As to the importance of projectivity, see \cite{Miyamoto}.
Two embeddings $R_{\CW} \hookrightarrow R_{L(k,0)}$ and
$A(\CW) \hookrightarrow A(L(k,0))$ are also obtained,
where $L(k,0) = L_{\widehat{sl}_2}(k,0)$.

In the second part, we discuss the parafermion vertex operator
algebra $K(\mathfrak{g},k)$ for a general $\mathfrak{g}$.
Using the $C_2$-cofiniteness of $K(sl_2,k)$ we show the following assertion 
(Theorem \ref{thm:C2_cofinte_general_case}).
\begin{enumerate}
\item[(5)] $K(\mathfrak{g},k)$ is $C_2$-cofinite for all finite dimensional simple
Lie algebras $\mathfrak{g}$ and all positive integers $k$.
\end{enumerate}

\medskip

Now let us give a brief discussion of our method. Denote by $\CN = N(sl_2,k)$ 
the commutant of the Heisenberg vertex operator algebra in the vacuum Weyl
module $V_{\widehat{sl}_2}(k,0)$. The vertex operator algebra
$\CN$ is not simple. It possesses a unique maximal ideal
$\mathcal{I}$ and $\CN/\mathcal{I} \cong \CW$.
Basic properties of $\CN$ and $\CW$ are established in \cite{DLWY, DLY}.
In fact, a set $\{ W^2, W^3, W^4, W^5 \}$ of strong generators,
the commutation relation among the operators $W^s_n$, $s = 2,3,4,5$, $n \in \Z$,
a null field $\bv^0$ of weight $8$ and a singular vector $\bu^0$ of weight $k+1$
are obtained.
Moreover, irreducible modules $M^{i,j}$ for $\CW$ are constructed
and the action of $o(W^s) = W^s_{s-1}$, $s = 2, 3, 4, 5$ on the top level of
$M^{i,j}$ is calculated.

The strong generators $W^2, W^3, W^4, W^5$ provide a
spanning set for $\CN$. The vectors of weight at most $7$ in the
spanning set are linearly independent. However, it is not the case
for those vectors of weight greater than $7$. There are two
nontrivial linear relations among the vectors of weight $8$ in the
spanning set. 
The null field $\bv^0$ gives one of the relations. 
The singular vector $\bu^0$, on the other hand, 
is closely related to the integrability condition
of the level $k$ integrable highest weight module
$L_{\widehat{sl}_2}(k,0)$ for $\widehat{sl}_2$.

Our main tool is a detailed analysis of the weight $8$ null field $\bv^0$
and the singular vector $\bu^0$,
together with  their images under the operator $W^3_1$ in Zhu's $C_2$-algebra.
Actually, we use $\bv^0$, $W^3_1 \bv^0$ and $(W^3_1)^2 \bv^0$
to show an embedding $R_\CN \hookrightarrow R_{V(k,0)}$
of Zhu's $C_2$-algebra  $R_\CN$ of $\CN$ in
Zhu's $C_2$-algebra $R_{V(k,0)}$ of $V(k,0) = V_{\widehat{sl}_2}(k,0)$.
The singular vector $\bu^0$ and $(W^3_1)^r \bu^0$, $r = 1, 2, 3$ are
necessary to determine the dimension of  $R_\CW$.
The calculation of various elements in $\CN$ or $V(k,0)$
is very difficult. On the other hand, the calculation in Zhu's $C_2$-algebra
is relatively easy, since the algebra is commutative and associative.
In fact, the operator $W^3_1$ induces a derivation on $R_\CN$ so that
we can express $(W^3_1)^r \bv^0$, $r = 0,1,2$ and
$(W^3_1)^r \bu^0$, $r = 0,1,2,3$ explicitly as elements of $R_{V(k,0)}$
by using the spanning set modulo $C_2(\CN)$ and
the embedding $R_\CN \hookrightarrow R_{V(k,0)}$.
This enable us to draw necessary information from these vectors.

In \cite{DW, DW2}, it was shown that the parafermion vertex operator algebra 
$K(\mathfrak{g},k)$ for a general $\mathfrak{g}$ is generated by
the vertex operator subalgebras $P_\alpha \cong K(sl_2, k_\alpha)$
corresponding to the positive roots $\alpha$, where $k_\alpha \in \{k, 2k, 3k \}$.
Furthermore, they proved the $C_2$-cofiniteness of $K(\mathfrak{g},k)$
under the assumption that each $P_\alpha$ is rational and $C_2$-cofinite.
In this paper we modify slightly their argument and
show that $K(\mathfrak{g},k)$ is $C_2$-cofinite
by using only the $C_2$-cofiniteness of $P_\alpha$'s.

If $k \le 4$, the vertex operator algebra $\CW$ is a well-known one.
In fact, $\CW $ is isomorphic to
an Ising model $\CL(1/2,0)$,
a $3$-state Potts model $\CL(4/5,0) \oplus \CL(4/5,3)$
or an orbifold $V_{\Z\beta}^+$ of a rank one lattice vertex operator algebra
$V_{\Z\beta}$ with $\la \beta,\beta \ra = 6$ according as $k = 2$, $3$ or $4$
\cite[Section 5]{DLY}.
Those vertex operator algebras are known to be rational and $C_2$-cofinite
and their irreducible modules are classified.
We note that $W^3 = W^4 = W^5 = 0$ if $k = 2$,
$W^4 = W^5 = 0$ if $k = 3$, and $W^5 = 0$ if $k = 4$ in $\CW$,
that is, these $W^s$'s are contained in the maximal ideal $\CI$ of $\CN$
and their images in $\CW = \CN/\CI$ is $0$
for such a small $k$.
Note also that  $\CW = \C$ if $k = 1$.
Thus we assume that $k \ge 5$ for the study of $\CW$.
In the case $k = 5, 6$, it is already known that $\CW$ is rational and $C_2$-cofinite, 
and the irreducible modules for $\CW$ are classified 
\cite[Section 5]{DLY}.

The organization of the paper is as follows.
Section \ref{sec:preliminaries} is devoted to preliminaries.
In Section \ref{sec:basic_properties_C2} we collect basic properties of
Zhu's $C_2$-algebra.
In Section \ref{sec:CN_mod_C2_CN} we study a set of generators of
$R_\CN$ and the action of  $W^3_1$ on it.
In Section \ref{sec:CN_mod_C2_Weyl_module} we discuss an embedding of $R_\CN$ in
$R_{V(k,0)}$ by using $(W^3_1)^r \bv^0$ modulo $C_2(\CN)$, $r = 0,1,2$.
Moreover, a differential operator corresponding to the action of $W^3_1$ is introduced.
In Section \ref{sec:singular_vec_and_CA} we calculate $(W^3_1)^r \bu^0$
modulo $C_2(V(k,0))$ explicitly, $r = 0,1,2,3$ and study the ideal of $R_{\CN}$
generated by these elements.
In Section \ref{sec:embedding_C2CW} we show an embedding of $R_{\CW}$ in
$R_{L(k,0)}$ and determine its dimension.
Here we use a basis of $L(k,0)$ \cite{FKLMM, MP}.
In Section  \ref{sec:Zhu_alg_CW} we show an embedding of $A(\CW)$ in $A(L(k,0))$
and determine its dimension. In particular, we have
$\dim R_{\CW} = \dim A(\CW) = k(k+1)/2$.
Furthermore, we show that $A(\CW)$ is semisimple and classify the irreducible
modules for $\CW$.
In Section \ref{sec:projectivity_CW} we discuss the projectivity of $\CW$.
Finally, in Section \ref{sec:C2_general_case} we modify the argument of
\cite{DW2} and show that the parafermion vertex operator algebra
$K(\mathfrak{g},k)$ is $C_2$-cofinite for all finite dimensional simple
Lie algebras $\mathfrak{g}$ and all positive integers $k$.

\section{Preliminaries}\label{sec:preliminaries}

In this section we fix notation and collect basic properties of various 
objects for later use. 
Our notation for vertex operator algebras is standard \cite{FLM, LL}. Let 
$V = (V,Y,\1,\omega)$ be a vertex operator algebra. 
Then $V = \bigoplus_{n \ge 0} V_{(n)}$ with $V_{(0)} = \C\1$ and 
$Y(v,z) = \sum_{n \in \Z} v_n z^{-n-1}$ is the vertex operator associated with $v \in V$. 
We also denote $\omega_{n+1}$ by $L(n)$. An element $v \in V_{(n)}$ is said to be 
homogeneous of weight $n$ and we write $\wt v = n$. For a  homogeneous $v$, 
we have $v_m V_{(n)} \subset V_{(n + \wt v - m - 1)}$, that is, the operator 
$v_m$ is of weight $\wt v - m - 1$.

We use the following identities for a vertex operator algebra $V$ 
\cite[(3.1.9), (3.1.12)]{LL}.
\begin{align}
[a_m, b_n] &= \sum_{i \ge 0} \binom{m}{i} (a_i b)_{m+n-i},\label{eq:VOA-formula-1}\\
(a_m b)_n &= \sum_{i \ge 0} (-1)^i \binom{m}{i}
\big( a_{m-i} b_{n+i} - (-1)^m b_{m+n-i} a_i \big),\label{eq:VOA-formula-2}\\
[L(-1), a_n] &=  (L(-1)a)_n = - n a_{n-1}\label{eq:Lm1-derivative}
\end{align}
for $a, b \in V$ and $m, n \in \Z$.

For a vertex operator algebra $V$, 
let $C_2(V) = \spn_{\C} \{ a_{-2}b \,|\, a, b \in V \}$. 
Zhu's $C_2$-algebra is the quotient space $R_V = V/C_2(V)$ 
\cite[Section 4.4]{Zhu}. 
The vertex operator algebra $V$ is said to 
be $C_2$-cofinite if $\dim R_V < \infty$.

Define two binary operations
\begin{equation}\label{eq:Zhu_binary_ops}
a*b = \sum_{i \ge 0} \binom{\wt a}{i} a_{i-1}b, \qquad
a \circ b = \sum_{i \ge 0} \binom{\wt a}{i}a_{i-2} b
\end{equation}
for homogeneous $a,b\in V$ and extend for
arbitrary $a, b \in V$ by linearity.
Let $O(V) = \spn_{\C} \{ a \circ b\,|\, a, b \in V\}$, which
is a two sided ideal with respect to the operation $*$.
Let $A(V)=V/O(V)$. Then
$(A(V),*)$ is an associative algebra \cite[Theorem 2.1.1]{Zhu} 
called Zhu's algebra of $V$.

A vertex operator algebra is said to be rational if every admissible module 
(that is, $\N$-graded weak module) is completely reducible.

From now through Section \ref{sec:projectivity_CW}, we fix an integer $k \ge 5$. 
As to parafermion vertex operator algebras associated with
$\widehat{sl}_2$, we tend to follow the notation in \cite{DLWY,
DLY}. In particular, $\{ h, e, f\}$ is a standard Chevalley basis
of $sl_2$ with $[h,e] = 2e$, $[h,f] = -2f$, $[e,f] = h$ for the
bracket, $\la \,\cdot\,,\,\cdot\,\ra$ is the normalized Killing form 
so that $\la h,h \ra = 2$, $\la e,f \ra = 1$, 
$\la h,e \ra = \la h,f \ra = \la e,e \ra = \la f,f \ra = 0$, and 
$\widehat{sl}_2 = sl_2 \otimes \C[t,t^{-1}] \oplus \C C$ 
is the corresponding affine Lie algebra.
Moreover, the vacuum Weyl module 
\begin{equation*}
V(k,0) = V_{\widehat{sl}_2}(k,0) = \Ind_{sl_2 \otimes \C[t]\oplus \C
C}^{\widehat{sl}_2}\C
\end{equation*}
with level $k$ is an induced $\widehat{sl}_2$-module such that $sl_2 \otimes \C[t]$ 
acts as $0$ and $C$ acts as $k$ on the vacuum vector $\1=1$. 
We denote by $a(n)$ the operator on $V(k,0)$ corresponding to the action of
$a \otimes t^n$. Then
\begin{equation}\label{eq:affine-commutation}
[a(m), b(n)] = [a,b](m+n) + m \la a,b \ra \delta_{m+n,0}k
\end{equation}
for $a, b \in sl_2$ and $m,n\in \Z$. Also $a(n)\1 =0$ for $n \ge 0$.
The vectors
\begin{equation}\label{eq:V-basis}
h(-i_1) \cdots h(-i_p) e(-j_1) \cdots e(-j_q) f(-m_1) \cdots f(-m_r)\1,
\end{equation}
$i_1 \ge \cdots \ge i_p \ge 1$, $j_1 \ge \cdots \ge j_q \ge 1$, $m_1 \ge
\cdots \ge m_r \ge 1$ and $p,q,r \ge 0$ form a basis of $V(k,0)$. 

Let $a(z) = \sum_{n \in \Z} a(n)z^{-n-1}$. Then $V(k,0)$ is a vertex
operator algebra such that $Y(a(-1)\1,z) = a(z)$ with central charge
$3k/(k+2)$ (\cite{FZ}, \cite[Section 6.2]{LL}). 
The conformal vector of $V(k,0)$ is 
\begin{equation*}
\omega_{\mraff} 
= \frac{1}{2(k+2)} \Big( -h(-2)\1 + \frac{1}{2}h(-1)^2\1 + 2e(-1)f(-1)\1 \Big).
\end{equation*}

The vector of \eqref{eq:V-basis} has weight
$i_1 + \cdots + i_p + j_1 + \cdots + j_q + m_1 + \cdots + m_r$
and it is also an eigenvector for $h(0)$ with eigenvalue $2(q-r)$. 
The vertex operator algebra $V(k,0)$ has a unique maximal ideal 
$\CJ$, which is generated by a single vector $e(-1)^{k+1}\1$ \cite{Kac}. 
The quotient space $L(k,0) = L_{\widehat{sl}_2}(k,0) = V(k,0)/\CJ$ is 
the integrable highest weight module for $\widehat{sl}_2$ with level $k$. 
For simplicity of notation we use the same symbol as \eqref{eq:V-basis} 
to denote its image in $L(k,0)$.

Let $M_{\widehat{\h}}(k,0)$ be the Heisenberg vertex operator algebra 
spanned by  $h(-i_1) \cdots h(-i_p)\1$, $i_1 \ge \cdots \ge i_p \ge 1$, $p \ge 0$ and 
\begin{equation*}
\CN = \{ v \in V(k,0)\,|\, h(n) v = 0 \text{ for } n \ge 0\}
\end{equation*}
be the commutant of  $M_{\widehat{\h}}(k,0)$ in $V(k,0)$. 
The conformal vector of $\CN$ is 
\begin{equation*}
W^2 = \frac{1}{2k(k+2)} \Big( -kh(-2)\1 -h(-1)^2\1 + 2k e(-1)f(-1)\1 \Big),
\end{equation*}
which is the difference of $\omega_{\mraff}$ and the conformal vector 
$(1/4k)h(-1)^2\1$ of $M_{\widehat{\h}}(k,0)$. 
The central charge of $\CN$ is $2(k-1)/(k+2)$. 
Note that the weight of a homogeneous element of $\CN$ coincides with its 
weight as an element of $V(k,0)$.

The vertex operator algebra $\CN$ is strongly generated by four vectors 
$W^2, W^3, W^4, W^5$ of weight $2,3,4,5$, respectively. 
That is, $\CN$ is spanned by the vectors
\begin{equation}\label{eq:normal-form}
W^2_{-i_1} \cdots W^2_{-i_p} W^3_{-j_1} \cdots W^3_{-j_q} W^4_{-m_1}
\cdots W^4_{-m_r} W^5_{-n_1} \cdots W^5_{-n_s}\1
\end{equation}
with $i_1 \ge \cdots \ge i_p \ge 1$, $j_1 \ge \cdots \ge j_q \ge 1$, $m_1
\ge \cdots \ge m_r \ge 1$, $n_1 \ge \cdots \ge n_s \ge 1$ 
and $p,q,r,s \ge 0$ 
(\cite[Theorem 3.1]{DLWY}, \cite[Lemma 2.4]{DLY}.) 
The vectors $W^3$, $W^4$ and $W^5$ are given explicitly in 
\cite[Appendix A]{DLY}. In particular, 
\begin{equation*}
\begin{split}
W^3 &= k^2 h(-3)\1 + 3 k h(-2)h(-1)\1 +
2h(-1)^3\1 - 6k h(-1)e(-1)f(-1)\1 \\
& \quad + 3 k^2e(-2)f(-1)\1 - 3 k^2e(-1)f(-2)\1.
\end{split}
\end{equation*}

The vector $W^s$ is up to a scalar multiple, 
a unique Virasoro primary vector of weight $s$ for $s = 3,4,5$. 
Thus $W^2_n W^s = 0$ for $n \ge 2$, $W^2_1 W^s = sW^s$ and 
$W^2_0 W^s = W^s_{-2}\1$. Moreover, $W^r_n W^s$, $3 \le r \le s \le 5$, $n \ge 0$ 
is expressed as a linear combination of vectors of the form \eqref{eq:normal-form} 
in \cite[Appendix B]{DLY}. Hence the commutation relation $[W^r_m, W^s_n]$ 
follows from the formulas \eqref{eq:VOA-formula-1} and \eqref{eq:VOA-formula-2}. 
The automorphism group 
$\Aut \CN = \langle \theta \rangle$ is of order $2$ with 
$\theta (W^s) = (-1)^s W^s$, $s = 2,3,4,5$ \cite[Theorem 2.5]{DLY}. 

The vertex operator algebra $\CN$ has a unique
maximal ideal $\CI$, which is generated by a singular vector 
 \cite[Theorem 4.2]{DLWY}
\begin{equation}\label{eq:singular-vec}
\bu^0 = f(0)^{k+1}e(-1)^{k+1}\1.
\end{equation} 
It is known that $\CI = \CN \cap \CJ$ \cite[Lemma 3.1]{DLY}. 
A parafermion vertex operator algebra
\begin{equation*}
\CW = K(sl_2,k) = \{ v \in L(k,0)\,|\, h(n) v = 0 \text{ for } n \ge 0\}
\end{equation*}
associated with $\widehat{sl}_2$ is the commutant 
of the Heisenberg vertex operator algebra $M_{\widehat{\h}}(k,0)$ in $L(k,0)$. 
It can be identified with the quotient $\CN/\CI$ of $\CN$ by $\CI$.  
The central charge of $\CW$ is $2(k-1)/(k+2)$ and 
$\CW = \bigoplus_{n \ge 0} \CW_{(n)}$ with $\CW_{(0)} = \C\1$ and 
$\CW_{(1)} = 0$. 
For simplicity of notation we use the same symbols $W^s$, $s = 2,3,4,5$ to 
denote their images in $\CW = \CN/\CI$.

An intermediate vertex operator algebra 
\begin{equation*}
V(k,0)(0) = \{ v \in V(k,0)\,|\, h(0) v = 0 \}
\end{equation*}
between $\CN$ and $V(k,0)$ 
was also considered in \cite{DLWY}. 
The set of vectors of the form \eqref{eq:V-basis} with $q=r$ is a basis of 
$V(k,0)(0)$ and $V(k,0)(0) = M_{\widehat{\h}}(k,0) \otimes \CN$ \cite[(2.5)]{DLY}. 
Although our main concern is the parafermion vertex operator algebra 
$\CW = \CN/\CI$, we need $V(k,0)(0)$ in our argument.

We should mention that the notation here is slightly different from that 
in \cite{DLWY, DLY}. In fact, $\CN$,  $\CW$ and $\CI$ are denoted by 
$N_0$, $K_0 = M^{0,0}$ and $\widetilde{\CI}$, respectively in \cite{DLWY, DLY} 
(see \cite[Theorems 3.1, 4.1 and 4.2]{DLWY}). 
The conformal vector $W^2$ of  $\CW$ is denoted by $\omega$ in 
\cite{DLWY, DLY}. 

For a positive number $\xi$, the symbol $[\xi]$ denotes the largest
integer which does not exceed $\xi$.

\section{Basic properties of Zhu's $C_2$-algebra}\label{sec:basic_properties_C2}

In this section we discuss basic properties of Zhu's $C_2$-algebra
\begin{equation*}
R_V = V/C_2(V)
\end{equation*}
of a vertex operator algebra $V$, which 
has  a commutative Poisson algebra structure by the operations 
$a \cdot b = a_{-1} b$ mod $C_2(V)$ and $\{a,b\} = a_0 b$ mod $C_2(V)$ 
\cite[Section 4.4]{Zhu}. 
The following properties are well-known.

\begin{lem}\label{lem:basic_C2}
$(1)$ $L(-1)V \subset C_2(V)$.

$(2)$ $a_n C_2(V) \subset C_2(V)$ for $n \le 0$, $a \in V$.

$(3)$ $a_n V \subset C_2(V)$ for $n \le 0$, $a \in C_2(V)$.

$(4)$ $a_1 b \equiv b_1 a \pmod{C_2(V)}$ for $a,b \in V$.
\end{lem}

\begin{lem}\label{lem:a1_derivation}
Let $a \in V$ be such that $a_0 V \subset C_2(V)$. 
Then $a_1 C_2(V) \subset C_2(V)$ and 
$a_1 (b_{-1} c) \equiv (a_1 b)_{-1} c + b_{-1}(a_1 c)$ mod $C_2(V)$ for $b, c \in V$. 
That is, $a_1$ induces a derivation of $R_V$ with respect to the product $a \cdot b$.
\end{lem}

\begin{proof}
It follows from \eqref{eq:VOA-formula-1} that 
$a_1 b_{-2}c =  (a_0 b)_{-1} c + (a_1 b)_{-2} c + b_{-2}a_1 c$. 
Moreover, $a_0 b \in C_2(V)$ by the assumption and so $(a_0 b)_{-1} c \in C_2(V)$ 
by Lemma \ref{lem:basic_C2}. Hence  $a_1 C_2(V) \subset C_2(V)$. 
Then the operator $a_1$ acts naturally on $R_V$ by 
$a_1 \cdot (b + C_2(V)) = a_1 b + C_2(V)$. 
We can verify the latter assertion by a similar argument as above.
\end{proof}

The conformal vector $\omega$ is a typical example 
which satisfies the assumption of the above lemma, for $\omega_0 = L(-1)$. 

We consider filtrations based on the weight subspaces $V_{(n)}$ of $V$. 
Let $F_p V = \bigoplus_{n \le p} V_{(n)}$, where $V_{(n)} = 0$ if $n < 0$. 
Then $F_{p-1} V \subset F_p V$, $\cap_{p} F_p V = 0$ and $\cup_{p} F_p V = V$. 
That is, the filtration $\{ F_p V \}$ is increasing, separated and exhaustive. 
Let $F_p A(V)$ be the image of $F_p V$ in $A(V)=V/O(V)$ 
\cite[Theorem 2.1.1]{Zhu} and set 
\begin{equation*}
\gr_p A(V) = F_p A(V)/F_{p-1} A(V), 
\end{equation*}
which is isomorphic to $(O(V) + F_p V)/(O(V) + F_{p-1} V)$. Let
\begin{equation*}
\gr A(V) = \bigoplus_{p} \gr_p A(V).
\end{equation*}

Every element of $\gr_p A(V)$ can be  written as 
$a + O(V) + F_{p-1} V$,  $a \in V_{(p)}$. 
For $u \in a + O(V) + F_{p-1} V$ with $a \in V_{(p)}$ and 
$v \in b + O(V) + F_{q-1} V$ with $b \in V_{(q)}$, we have 
\begin{equation}\label{eq:prod_in_gr_AV}
u*v \in a_{-1}b + O(V) + F_{p+q-1}V
\end{equation}
by \eqref{eq:Zhu_binary_ops}. 
Moreover, \cite[Lemma 2.1.3]{Zhu} implies that 
\begin{equation}\label{eq:commutation_in_gr_AV}
u*v - v*u \in a_0 b + O(V) + F_{p+q-2} V.
\end{equation}
In particular, $u*v = v*u$ in $\gr_{p+q} A(V)$ since $\wt a_0 b = p+q-1$. 
Thus $\gr A(V)$ is a $\Z$-graded commutative Poisson algebra 
with respect to the product $u*v$ and the commutation $u*v - v*u$. 
We also have that $\gr A(V) \cong A(V)$ as vector spaces.

As to $R_V=V/C_2(V)$, note that the weight space decomposition 
$C_2(V) = \bigoplus_n C_2(V)_{(n)}$ holds, where 
$C_2(V)_{(n)} = C_2(V) \cap V_{(n)}$. 
This  decomposition induces a natural $\Z$-grading
\begin{equation}\label{Z-grading-RV}
R_V = \bigoplus_p (R_V)_p
\end{equation}
of $R_V $, where 
\begin{equation*}
 (R_V)_p = V_{(p)}/C_2(V)_{(p)}.
\end{equation*}

For $u \in a + C_2(V)_{(p)}$ with $a \in V_{(p)}$ and 
$v \in b + C_2(V)_{(q)}$ with $b \in V_{(q)}$, we have 
\begin{equation}\label{eq:prod_in_gr_RV}
u_{-1} v \in a_{-1} b + C_2(V)_{(p+q)}
\end{equation}
and
\begin{equation}\label{eq:commutation_in_gr_RV}
u_0 v \in a_0 b + C_2(V)_{(p+q-1)}.
\end{equation}
Since $a_{-1} b \in V_{(p+q)}$ and $a_0 b \in V_{(p+q-1)}$, we see that 
$R_V = \bigoplus_p (R_V)_p$ is a $\Z$-graded commutative Poisson algebra 
with respect to the operations $u \cdot v = u_{-1} v$ and $\{ u,v \} = u_0 v$.

\begin{prop}\label{prop:gr_AV_onto_gr_RV}
$C_2(V)_{(p)} \subset O(V) + F_{p-1} V$ for all $p$ and the assignment 
\begin{equation*}
a + C_2(V)_{(p)} \mapsto a + O(V) + F_{p-1} V
\end{equation*}
for $a \in V_{(p)}$ defines a surjective homomorphism 
$R_V \rightarrow \gr A(V)$ of $\Z$-graded Poisson algebras.
\end{prop}

\begin{proof}
The space $C_2(V)_{(p)}$ is spanned by the weight $p$ 
elements of the form $a_{-2} b$ with homogeneous $a, b$. 
For such $a, b$, we have $\wt a_{i-2} b \le p-1$ if $i \ge 1$. 
Hence $a_{-2} b \equiv a \circ b$ mod $F_{p-1} V$ by 
\eqref{eq:Zhu_binary_ops}. Thus $C_2(V)_{(p)} \subset O(V) + F_{p-1} V$  
and the assignment is well-defined. 
Now, \eqref{eq:prod_in_gr_AV}, \eqref{eq:commutation_in_gr_AV}, 
\eqref{eq:prod_in_gr_RV} and \eqref{eq:commutation_in_gr_RV} imply that 
the map $R_V \rightarrow \gr A(V)$ defined by the 
assignment is a homomorphism of  $\Z$-graded Poisson algebras.
\end{proof}

\section{Zhu's $C_2$-algebra $R_{\CN}$ of $\CN$}\label{sec:CN_mod_C2_CN}

Since $\CW = \CN/\mathcal{I}$, we have $C_2(\CW) = (C_2(\CN)+\CI)/\CI$ and
$\CW/C_2(\CW) \cong \CN/(C_2(\CN)+\CI)$.
In this section we study $R_{\CN} = \CN/C_2(\CN)$, which is infinite dimensional. 
Recall that $\{ W^2, W^3, W^4, W^5 \}$ is a set of strong generators for the 
vertex operator algebra $\CN$. Thus $R_{\CN}$ is generated by
$\widetilde{W}^s = W^s + C_2(\CN)$, $s = 2,3,4,5$ as a 
commutative associative algebra, and it is a homomorphic image 
\begin{equation*}
\varphi : \mathbb{C}[x_2,x_3,x_4,x_5] \rightarrow R_{\CN}; \quad
x_s \mapsto \widetilde{W}^s = W^s + C_2(\CN)
\end{equation*}
of a polynomial algebra with four variables $x_2,x_3,x_4,x_5$ 
($\varphi$ is denoted by $\tilde{\rho}$ in \cite{DLY}.)

\begin{lem}\label{lem:wt_theta_condition}
If $v$ is a homogeneous element of  $\CN$ such that 
$\wt v$ is odd and $\theta(v) = v$ or $\wt v$ is even and $\theta(v) = -v$, 
then $v \in C_2(\CN)$.
\end{lem}

\begin{proof}
Since $\CN$ is spanned by the vectors of the form \eqref{eq:normal-form} 
and since those vectors are eigenvectors for $\theta$ with eigenvalue 
$\pm 1$, we may assume that $v$ is of the form \eqref{eq:normal-form}. 
Such a vector $v$ is contained in $C_2(\CN)$ unless 
$i_1, \ldots, i_p$, $j_1, \ldots, j_q$, $m_1, \ldots, m_r$, 
$n_1, \cdots, n_s$ are all $1$. Thus it is sufficient to discuss the case 
$v = (W^2_{-1})^p  (W^3_{-1})^q  (W^4_{-1})^r  (W^5_{-1})^s \1$. 
In this case $\wt v = 2p + 3q + 4r + 5s$ and $\theta(v) = (-1)^{q+s} v$. 
Hence $\theta(v) = v$ if and only if $\wt v$ is even. This proves the lemma.
\end{proof}

The vector $W^r_{2m} W^s$ is of weight $r + s - 2m -1 $ and furthermore  
it is an eigenvector for $\theta$ with eigenvalue $(-1)^{r+s}$. 
Hence $W^r_{2m} W^s \in C_2(\CN)$, $r,s = 2,3,4,5$ by 
Lemma \ref{lem:wt_theta_condition}. 

\begin{lem}\label{lem:Wr0_CN}
$W^r_0 \CN \subset C_2(\CN)$, $r = 2,3,4,5$.
\end{lem}

\begin{proof}
It is sufficient to show $W^r_0 v \in  C_2(\CN)$ for a vector $v$ of the form 
\eqref{eq:normal-form}. 
We have 
$a_0 b_n c = (a_0 b)_n c + b_n a_0 c$ by \eqref{eq:VOA-formula-1}. 
Now,  $W^r_0 W^s \in C_2(\CN)$ for $s = 2,3,4,5$ and $W^r_0 \1 = 0$. 
Hence $W^r_0 v \in  C_2(\CN)$ by Lemma \ref{lem:basic_C2}, as required.
\end{proof}

The following lemma is a direct consequence of 
Lemmas \ref{lem:a1_derivation} and \ref{lem:Wr0_CN}.

\begin{lem}\label{Wr1_mod_C2_of_CN}
For $r = 2,3,4,5$, we have 
$W^r_1 C_2(\CN) \subset C_2(\CN)$ and the action of $W^r_1$ on $R_{\CN}$ 
defined by $W^r_1 \cdot (u + C_2(\CN)) = W^r_1 u + C_2(\CN)$ for $u \in \CN$ 
is a derivation on the commutative associative algebra $R_\CN$ generated 
by $\widetilde{W}^s$, $s = 2,3,4,5$. 
\end{lem}

Since $W^2$ is the conformal vector of $\CN$, the assertions of the above 
two lemmas for $r = 2$ are general facts for a vertex operator algebra. 

As to the singular vector $\bu^0$ of \eqref{eq:singular-vec}, 
we have $\wt \bu^0 = k+1$ and $\theta(\bu^0) = (-1)^{k+1} \bu^0$ \cite[Theorem 4.2]{DLWY}. 
Hence $W^r_{2m} \bu^0 \in C_2(\CN)$, $m \in \Z$, $r = 2,3,4,5$ 
by Lemma \ref{lem:wt_theta_condition}. 
Of course, $W^r_n \bu^0 = 0$, $n \ge r$, for $\bu^0$ is a singular vector of $\CN$.  
Also $W^r_{2m} \bu^0 \in C_2(\CN)$ if $m < 0$ by the definition of $C_2(\CN)$.

The action of $W^3_1$ will play an important role in our argument. 
We know the action of  $W^3_1$ on $\widetilde{W}^s$ by
\cite[ Appendix B]{DLY}.
\begin{lem}\label{lem:action-of-W31-on-Ws}
The action of $W^3_1$ on $\widetilde{W}^s$, $s = 2,3,4,5$ is as follows.
\begin{align*}
W^3_1 \cdot \widetilde{W}^2
&= 3 \widetilde{W}^3,\\
W^3_1 \cdot \widetilde{W}^3
&= \frac{1}{16 k+17}\big( 288 k^3 (k-2) (k+2)^2 (3 k+4)(\widetilde{W}^2)^2
+ 36 k (2 k+3)\widetilde{W}^4 \big),\\
W^3_1 \cdot \widetilde{W}^4
&= \frac{1}{64 k+107}\big( 1248 k^2 (k-3) (k+2) (2 k+1) (2 k+3) \widetilde{W}^2\widetilde{W}^3\\
&\  - 12 k (3 k+4) (16 k+17) \widetilde{W}^5 \big),\\
W^3_1 \cdot \widetilde{W}^5
&= \frac{240 k^4 (k+2)^3 (2 k+3) (3 k+4) (202 k-169)}{16 k+17} (\widetilde{W}^2)^3\\
&\  - 15 k (2 k+3) (41 k+61) (\widetilde{W}^3)^2\\
&\ + \frac{60 k^2 (k+2) (404 k^2+1170 k+835)}{16 k+17} \widetilde{W}^2 \widetilde{W}^4.
\end{align*}
\end{lem}

The above lemma implies that $R_\CN$ is generated by 
$(W^3_1)^r \cdot \widetilde{W}^2$, $r = 0,1,2,3$ as an associative algebra.

In \cite[Appendix B]{DLY}, the vectors $W^r_n W^s$, $3 \le r \le s \le 5$, 
$n \ge 0$ are expressed as linear combinations of the vectors of the form 
\eqref{eq:normal-form} with coefficients being rational functions of $k$. 
However, it is hard to calculate directly the action of $W^3_1$ on a vector of 
\eqref{eq:normal-form} by using the formulas 
\eqref{eq:VOA-formula-1} and \eqref{eq:VOA-formula-2}, because 
it is very complicated. 
On the other hand the above argument implies that the action of $W^3_1$ 
can be manageable modulo $C_2(\CN)$, that is, in $R_{\CN}$. 
We will pursue this in the subsequent section.

\section{Embedding of $R_\CN$ in $R_{V(k,0)}$}\label{sec:CN_mod_C2_Weyl_module}

In this section we study $\CN$ modulo $C_2(V(k,0))$. 
Since the set of vectors of the form \eqref{eq:V-basis} is a basis of $V(k,0)$, 
Zhu's $C_2$-algebra $R_{V(k,0)} = V(k,0)/C_2(V(k,0))$ of $V(k,0)$ is 
a polynomial algebra $\mathbb{C}[y_0,y_1,y_2]$ with three variables
\begin{equation}\label{eq:def-y0y1y2}
y_0 = \overline{h(-1)\mathbf {1}}, 
\quad y_1 = \overline{e(-1)\mathbf {1}},
\quad y_2 = \overline{f(-1)\mathbf {1}},
\end{equation}
where $\overline{v} = v + C_2(V(k,0))$ is the image of $v \in V(k,0)$ in $R_{V(k,0)}$.
Since $C_2(\CN) \subset C_2(V(k,0))$, there is a natural homomorphism 
\begin{equation*}
\psi : R_{\CN} \rightarrow R_{V(k,0)}; \quad 
\widetilde{v} = v + C_2(\CN) \mapsto \overline{v} = v + C_2(V(k,0))
\end{equation*}
of commutative Poisson algebras. Let
\begin{equation*}
\CA = \CN +  C_2(V(k,0))
\end{equation*}
be its image and set $y = y_0$, $z = y_1 y_2$. 
Then $\CA \subset \mathbb{C}[y,z]$.
In fact, $\CN \subset V(k,0)(0)$ and the image of $V(k,0)(0)$ in 
$R_{V(k,0)}$ is $\C[y,z]$. 
More precisely, 
$\overline{W^s} = W^s + C_2(V(k,0)) \in R_{V(k,0)}$, $s=2,3,4,5$
are as follows \cite[Appendix A]{DLY}.
\begin{align*}
\overline{W^2} &=
-\frac{1}{2 k (k+2)} (y^2 - 2kz), \\
\overline{W^3} &= 2 (y^3 - 3 k y z),\\
\overline{W^4} &=
-(11 k+6) y^4 + 4 k (11 k+6) y^2 z - 2 k^2 (6 k-5) z^2,\\
\overline{W^5} &=
-2 (19 k+12) y^5 + 10 k (19 k+12) y^3 z - 10 k^2 (10 k-7) y z^2.
\end{align*}
Note that 
$\psi: R_{\CN} \rightarrow \CA$; $W^s + C_2(\CN) \mapsto W^s + C_2(V(k,0))$ 
and the subalgebra $\CA$ of $\C[y,z]$ is generated by 
$\overline{W^s}$, $s=2,3,4,5$.

We assign the weight to each term of $\mathbb{C}[y,z]$ by $\mathrm{wt}\, y = 1$ and
$\mathrm{wt}\, z = 2$. Since the weight of the elements 
$h(-1)\mathbf {1}$ and $e(-1)f(-1)\mathbf {1}$
of $V(k,0)$ is $1$ and $2$, respectively, the weight on  $\mathbb{C}[y,z]$ 
is compatible with that on $V(k,0)$ and so on $R_{V(k,0)}$.

Consider the composition $\psi \circ \varphi$ of the homomorphisms 
$\varphi$ and $\psi$.
\begin{equation*}
\psi \circ \varphi : \mathbb{C}[x_2,x_3,x_4,x_5] \rightarrow R_{\CN}
\rightarrow \CA; \quad
x_s \mapsto \widetilde{W}^s \mapsto \overline{W^s}.
\end{equation*}
The kernel $\Ker \psi \circ \varphi$ of the composition
$\psi \circ \varphi$ is the algebraic relations among
$\overline{W^s}$, $s = 2,3,4,5$ in the polynomial algebra $\mathbb{C}[y,z]$. 
In \cite[Section 2]{DLY}, three polynomials 
$B_0, B_1, B_2 \in \mathbb{C}[x_2,x_3,x_4,x_5]$
corresponding to null fields $\mathbf{v}^0, \mathbf{v}^1,\mathbf{v}^2$ are introduced.
In fact, $\varphi(B_i)$ is a nonzero constant multiple of 
$\widetilde{\mathbf{v}}^i = \bv^i + C_2(\CN)$, 
$i = 0,1,2$. Since $\mathbf{v}^i = 0$ in $\CN$, it follows that $\varphi(B_i) = 0$.

At this stage, using a computer algebra system Risa/Asir, we can verify that
the kernel $\Ker \psi \circ \varphi$ coincides with the ideal
$\langle B_0, B_1, B_2 \rangle$ of $ \mathbb{C}[x_2,x_3,x_4,x_5]$ generated by
$B_0, B_1, B_2$.
This implies that the kernel of $\varphi$ is also $\langle B_0, B_1, B_2 \rangle$
and $\psi$ is an isomorphism. Thus the following theorem holds. 

\begin{thm}\label{thm:embedding_RN_in_RV}
$(1)$ $R_{\CN} \cong \mathbb{C}[x_2,x_3,x_4,x_5]/\langle B_0, B_1, B_2 \rangle$.

$(2)$ $C_2(\CN) = \CN \cap C_2(V(k,0))$ and $R_{\CN} \cong \CA$.
\end{thm}

We make some remarks on null fields. 
Recall that a null field is a nontrivial linear relations among the vectors 
of the form \eqref{eq:normal-form}. 
There are two linearly independent weight $8$ null fields in $\CN$, 
one is fixed by the automorphism $\theta$ of $\CN$ and the other is 
an eigenvector for $\theta$ with eigenvalue $-1$ \cite[Appendix C]{DLY}. 
This latter one is a linear combination of vectors 
of the form \eqref{eq:normal-form} contained in $C_2(\CN)$ 
by Lemma \ref{lem:wt_theta_condition}, and so its image in $R_\CN$ 
does not yield a nontrivial relation in $R_\CN$. 

The null field $\bv^0$ is, up to scalar multiple a unique weight $8$ null field 
which is fixed by $\theta$. 
On the other hand, the choice of the null fields $\bv^1, \bv^2$ 
of \cite[Section 2]{DLY} is not intrinsic. 
Now, consider $W^3_1 \bv^0$ and $(W^3_1)^2 \bv^0$. 
The null field $\bv^0$ is given explicitly in \cite[Appendix C]{DLY}. 
Using Lemma \ref{lem:action-of-W31-on-Ws}, we can calculate 
the image $(W^3_1)^r \bv^0 + C_2(\CN)$ of $(W^3_1)^r \bv^0$ in $R_\CN$. 
The vectors $W^3_1 \bv^0$ and $(W^3_1)^2 \bv^0$ are 
null fields of weight $9$ and $10$, respectively. 
We also define an action of $W^3_1$ on $\C[x_2,x_3,x_4,x_5]$ so that 
it induces a derivation and acts on $x_s$ just as $W^3_1 \cdot \widetilde{W}^s$ 
described in Lemma \ref{lem:action-of-W31-on-Ws}, $s=2,3,4,5$. 
That is, the action of  $W^3_1$ on $\C[x_2,x_3,x_4,x_5]$ is compatible with 
the action on $R_{\CN}$ under the correspondence 
$x_s \leftrightarrow \widetilde{W}^s$. 
Let $T_0 = B_0$, $T_1 = W^3_1 \cdot T_0$ and $T_2 = (W^3_1)^2 \cdot T_0$, 
so that $\varphi (T_r) =  (W^3_1)^r \cdot \widetilde{\bv}^0$. 
We can verify that $B_1$ is a scalar multiple of $T_1$ and 
$B_2$ is a linear combination of $x_2 T_0$ and $T_2$. 
Thus the ideal $\langle T_0, T_1, T_2 \rangle$ of $\C[x_2,x_3,x_4,x_5]$ 
generated by $T_0, T_1, T_2$ coincides with  $\langle B_0, B_1, B_2 \rangle$. 
Although $\CN$ has infinitely many null fields, 
Theorem \ref{thm:embedding_RN_in_RV} implies that 
the three null fields $\bv^0$, $W^3_1 \bv^0$ and $(W^3_1)^2 \bv^0$ 
are sufficient modulo $C_2(\CN)$.

We replace $\overline{W^s}$ with simpler polynomials in $\C[y,z]$, $s = 2,3,4,5$,
namely, we set
\begin{equation}\label{eq:def-gs}
g_2 = y^2 - 2kz, \quad  g_3 = y^3 - 3kyz, \quad
g_4 = z^2, \quad g_5 = yz^2.
\end{equation}
Then $g_s$ is a homogeneous polynomial of weight $s$. 
Indeed, we can express $\overline{W^s}$ as follows.
\begin{align*}
\overline{W^2} &= - \frac{1}{2k(k+2)}g_2,\\
\overline{W^3} &= 2g_3,\\
\overline{W^4} &= -(11k+6)g_2^2 + 2k^2(16k+17)g_4,\\
\overline{W^5} &= -2(19k+12)g_2 g_3 + 2k^2(64k+107)g_5.
\end{align*}
Conversely,
\begin{align*}
g_2 &= -2k(k+2)\overline{W^2},\\
g_3 &= \frac{1}{2}\overline{W^3},\\
g_4 &= \frac{2(k+2)^2(11k+6)}{16k+17}(\overline{W^2})^2
+ \frac{1}{2k^2(16k+17)}\overline{W^4},\\
g_5 &= -\frac{(k+2)(19k+12)}{k(64k+107)}\overline{W^2} \cdot \overline{W^3}
+ \frac{1}{2k^2(64k+107)}\overline{W^5}.
\end{align*}

Therefore,  $\CA$ is generated by $g_s$, $s=2,3,4,5$
as a subalgebra of $\C[y,z]$. 
We note that
\begin{align*}
g_2^4 - g_2 g_3^2 - 5 k^2 g_2^2 g_4 + 4 k^4 g_4^2 + 2 k^2 g_3 g_5 &= 0,\\
g_2^3 g_3 - g_3^3 - 5 k^2 g_2 g_3 g_4 + 2 k^2 g_2^2 g_5 - 2 k^4 g_4 g_5 &= 0,\\
g_2^3 g_4 - g_3^2 g_4 - 4 k^2 g_2 g_4^2 + k^2 g_5^2 &= 0
\end{align*}
are the algebraic relations among  $g_2, g_3, g_4, g_5$ in $\C[y,z]$.
That is, the kernel of an algebra homomorphism from the polynomial algebra
$\C[t_2,t_3,t_4,t_5]$ with four variables into $\C[y,z]$ which maps $t_s$ to
$g_s$, $s=2,3,4,5$ is generated by the three polynomials of the left hand side of
the above three equations. 
The kernel of the homomorphism $t_s \mapsto g_s$,  $s=2,3,4,5$ corresponds to 
the kernel of $\psi \circ \varphi: x_s \mapsto \overline{W^s}$,  $s=2,3,4,5$ by 
$g_s \leftrightarrow \overline{W^s}$. The algebraic relations among  
$g_2, g_3, g_4, g_5$ were calculated by the computer algebra system Risa/Asir. 

We can verify that the terms $g_2^p g_3^q g_4^r g_5^s$ of weight 
at most $7$ are linearly independent polynomials in $\mathbb{C}[y,z]$. 
We also have
\begin{align*}
z^2 &= g_4, \\
yz^2 &= g_5, \\
z^3 &= \frac{3}{2k}g_2 g_4 - \frac{1}{2k^3}(g_2^3 - g_3^2), \\
yz^3 &= \frac{1}{k}(g_2 g_5 - g_3 g_4).
\end{align*}

Denote by $\C[y,z]_{(n)}$ and $\CA_{(n)}$ the weight $n$ subspaces of
$\C[y,z]$ and $\CA$ with respect to the weight defined by
$\wt y = 1$ and $\wt z = 2$. 
Recall that $[n/2]$ is the largest integer which does not exceed $n/2$. 

\begin{lem}\label{lem:basis_CA}
$(1)$ $\dim \C[y,z]_{(n)}  = [n/2] + 1$ with $\{y^{n-2j}z^j;\, 0 \le j \le [n/2]\}$
a basis of $\C[y,z]_{(n)}$ for $n \ge 0$.

$(2)$ $\CA_{(0)} = \C$, $\CA_{(1)} = 0$, and
$\dim \CA_{(n)} = [n/2]$ with $\{y^n-nky^{n-2}z, y^{n-2j}z^j;\, 2 \le j \le [n/2]\}$
a basis of $\mathcal{A}_{(n)}$ for $n \ge 2$. 
In particular, $\C[y,z]_{(n)}  = \C y^n \oplus \CA_{(n)}$ for $n \ge 1$.
\end{lem}

\begin{proof}
The assertion (1) is clear from the definition of weight on $\C[y,z]$.
Since $\CA$ is generated by $g_2, g_3, g_4, g_5$, we see that $\CA$ contains $z^j$ and
$yz^j$ for $j \ge 2$. Moreover, $y^i z^j = y^{i-2} z^j g_2 + 2k y^{i-2}z^{j+1}$ 
for $i \ge 2$ and so
$y^i z^j \in \CA$ for $i \ge 0$, $j \ge 2$. Then the assertion (2) follows from
the expansion of the term $g_2^p g_3^q$.
\end{proof}

We define an action $W^3_1 \cdot f(y,z)$ of $W^3_1$ on
$f(y,z) \in \mathcal{A}$. Note that
$C_2(V(k,0))$ is not invariant under the operator $W^3_1$. For instance, 
\begin{equation*}
\begin{split}
W^3_1 e(-2)\1 
&= - 3(5k^2 - 6k - 16) h(-2) e(-1)\1 - 3(7k^2 - 2k - 8) h(-1) e(-2)\1\\
& \quad + 6(k+2) h(-1)^2 e(-1)\1 - 12k e(-1)^2 f(-1)\1\\
& \quad  + 6k(k-2)(5k+8) e(-3)\1 
\end{split}
\end{equation*}
is not contained in $C_2(V(k,0))$. 
Since this is the case, we first consider the action of $W^3_1$ on $R_\CN$ as discussed 
in Section \ref{sec:CN_mod_C2_CN} and then transform it to  $\mathcal{A}$ by
the isomorphism $\psi$. 
Thus the action of $W^3_1$ on $\overline{v}$ is 
$W^3_1 \cdot \overline{v} = \psi(W^3_1 \cdot \widetilde{v})$ 
for $v \in \CN$, that is, 
\begin{equation*}
W^3_1 \cdot (v + C_2(V(k,0))) = W^3_1 v + C_2(V(k,0)).
\end{equation*}
In particular, 
$W^3_1 \cdot \overline{W^s} = \psi(W^3_1 \cdot \widetilde{W}^s)$, $s = 2,3,4,5$. 

The following lemma is a consequence of 
Lemma \ref{lem:action-of-W31-on-Ws}
and the relations between
$\overline{W^2}$, $\overline{W^3}$, $\overline{W^4}$, $\overline{W^5}$
and $g_2$, $g_3$, $g_4$, $g_5$.
\begin{lem}\label{lem:W3_1_on_gs}
$W^3_1$ acts on $g_s$, $s=2,3,4,5$ as follows.
\begin{align*}
W^3_1 \cdot g_2 &= -12 k (k+2) g_3,\\
W^3_1 \cdot g_3 &= -18 k (k+2) g_2^2 + 36 k^3 (2 k+3) g_4,\\
W^3_1 \cdot g_4 &= -12 k (3 k+4) g_5,\\
W^3_1 \cdot g_5 &= \frac{6 (7 k+9)}{k} (g_2^3 - g_3^2) - 6 k (28 k+37) g_2 g_4.
\end{align*}
\end{lem}

We also note that $W^3_1$ acts on a polynomial in $g_s$, $s=2,3,4,5$ as a derivation,
since it acts on a polynomial in $\widetilde{W}^s$, $s=2,3,4,5$ similarly.

Define a differential operator $D$ on $\mathbb{C}[y,z]$ by
\begin{equation}\label{eq:def-D}
D = \big( (k+2)y^2 - 2kz \big) \frac{\partial}{\partial y} + (3k+4)yz \frac{\partial}{\partial z}.
\end{equation}

\begin{prop}\label{prop:D_and_W3p1}
The restriction of the action of $-6kD$ to $\mathcal{A}$ coincides with the action of
$W^3_1$ on  $\mathcal{A}$.
\end{prop}

\begin{proof}
Calculate the action of $D$ on the polynomials $g_s$, $s=2,3,4,5$ of \eqref{eq:def-gs} 
and compare it with Lemma \ref{lem:W3_1_on_gs}. Then we obtain the assertion. 
\end{proof}

Since $W^3$ lies in the commutant $\CN$ of the Heisenberg vertex operator algebra 
$M_{\widehat{\h}}(k,0)$, it follows that $W^3_1 h(-1)\1 = h(-1) W^3_1 \1 = 0$. 
On the other hand, the action of $D$ on $y = \overline{h(-1)\1}$ is nonzero. 
We note that the differential operator $D$ is obtained by means of the action of $W^3_1$ 
on $R_{\CN}$ and the isomorphism $\psi$.

\section{Singular vector and ideals of $\mathcal{A}$}\label{sec:singular_vec_and_CA}

In this section we study the image $\overline{\mathbf{u}^0} = \mathbf{u}^0 + C_2(V(k,0))$ 
of the singular vector $\mathbf{u}^0$ defined by \eqref{eq:singular-vec} 
in $R_{V(k,0)} = V(k,0)/C_2(V(k,0))$. 
We also consider $(W^3_1)^r \cdot \overline{\mathbf{u}^0}$, $r =1,2,3$.

For $s \ge 2$ and $1 \le j \le [s/2]$, set
\begin{equation}\label{eq:def-Qsj}
Q(s,j) = \sum_{\substack{1 \le i_1 < i_2 < \cdots < i_j \le s\\
2 \le i_1, 4 \le i_2, \ldots, 2j \le i_j}}
(i_1-1)(i_2-3) \cdots (i_j-2j+1).
\end{equation}

\begin{lem}\label{lem:Qsj}
For $s \ge 2$ and $1 \le j \le [s/2]$, we have
\begin{equation}\label{eq:formula-Qsj}
Q(s,j) = \frac{s !}{2^j (s-2j)! j!}.
\end{equation}
\end{lem}

\begin{proof}
We proceed by induction on $j$.
The equation is clear if $j = 1$,
for both side of \eqref{eq:formula-Qsj} are equal to $(s-1)s/2$ in
this case.

Assume that the equation \eqref{eq:formula-Qsj} holds for $j-1$.
Let $t = i_j$. Then $Q(s,j)$ can be written as
\begin{equation*}
Q(s,j) = \sum_{t=2j}^s (t-2j+1)Q(t-1,j-1).
\end{equation*}

By the  induction hypothesis, we have
\begin{equation*}
Q(t-1,j-1) = \frac{(t-1)!}{2^{j-1} (t-2j+1)! (j-1)!}
\end{equation*}
and so
\begin{equation*}
Q(s,j) = \frac{1}{2^{j-1} (j-1)!} \sum_{t=2j}^s (t-2j+1)(t-2j+2) \cdots (t-1).
\end{equation*}
Since
\begin{align*}
& (t-2j+1)(t-2j+2) \cdots (t-1) t -
(t-2j)(t-2j+1)(t-2j+2) \cdots (t-1)\\
&\qquad = 2j  (t-2j+1)(t-2j+2) \cdots (t-1),
\end{align*}
we can calculate that
\begin{equation*}
\sum_{t=2j}^s (t-2j+1)(t-2j+2) \cdots (t-1) 
= \frac{1}{2j} (s-2j+1)(s-2j+2) \cdots (s-1)s
\end{equation*}
Hence \eqref{eq:formula-Qsj} holds for $j$ and the induction is complete.
\end{proof}

Recall that
$R_{V(k,0)} = \C[y_0,y_1,y_2]$ with three variables
$y_0$, $y_1$, $y_2$ defined by \eqref{eq:def-y0y1y2}. 
The subspace $C_2(V(k,0))$ is invariant under the operator $f(0)$, and
in fact $f(0)$ induces a derivation on $\C[y_0,y_1,y_2]$ by the action 
$f(0) \cdot \overline{v} = \overline{f(0)v}$. 
More precisely, the following lemma holds since $[f(0), h(-1)] = 2f(-1)$, 
$[f(0), e(-1)] = -h(-1)$ and $[f(0), f(-1)] = 0$.

\begin{lem}\label{eq:f0-action-1}
For $p, q, r \ge 0$, we have
\begin{align*}
f(0) \cdot y_0^p y_1^q y_2^r
&= \overline{ f(0) h(-1)^p e(-1)^q f(-1)^r \mathbf{1}}\\
&= 2p y_0^{p-1} y_1^q y_2^{r+1}
- q y_0^{p+1} y_1^{q-1} y_2^r.
\end{align*}
\end{lem}

The above lemma implies that $f(0)$ acts on $R_{V(k,0)} = \C[y_0,y_1,y_2]$ 
as a differential operator
\begin{equation}\label{eq:def-CD}
2 y_2 \frac{\partial}{\partial y_0} - y_0 \frac{\partial}{\partial y_1}.
\end{equation}

We will calculate 
$\overline{f(0)^s e(-1)^n \1} =
(2 y_2 \frac{\partial}{\partial y_0} - y_0 \frac{\partial}{\partial y_1})^s y_1^n$
for $n \ge 1$ and $0 \le s \le 2n$.
Note that $f(0)^s e(-1)^n \1 = 0$ if $s > 2n$.

\begin{thm}\label{thm:f0-action-mod-C2-general}
For $n \ge 1$ and $0 \le s \le 2n$,
\begin{equation}\label{eq:F1}
\overline{f(0)^s e(-1)^n \1} =
\sum_{j = \max\{0, s-n\}}^{[s/2]} c_j^{(s,n)} y_0^{s-2j} y_1^{n-s+j} y_2^j,
\end{equation}
where
\begin{equation}\label{eq:F2}
c_j^{(s,n)} = (-1)^{s-j}\frac{s!\,n!}{(s-2j)!\,(n-s+j)!\,j!}
\end{equation}
for $\max\{0, s-n\} \le j \le [s/2]$.
Here $[s/2]$ denotes the largest integer which does not exceed $s/2$.
\end{thm}

\begin{proof}
Set $\bv = f(0)^s e(-1)^n \1$ for simplicity of notation.
We first show \eqref{eq:F1}. We examine which terms $y_0^p y_1^q y_2^r$ appear in
the expression of $\overline{\bv}$ as a polynomial in $\C[y_0,y_1,y_2]$. The weight of
$\bv$ is $n$ and so
\begin{equation*}
p+q+r = n.
\end{equation*}

The vector $h(-1)^p e(-1)^q f(-1)^r \mathbf{1}$ is an eigenvector for $h(0)$ with
eigenvalue $2(q-r)$. Since $h(0)\bv = 2(n-s)\bv$, it follows that 
\begin{equation*}
q-r = n-s.
\end{equation*}
Form these two equations we have $q = n-s+r$ and $p=s-2r$. Replace $r$ with $j$.
Then we see that only the terms of the form $y_0^{s-2j} y_1^{n-s+j} y_2^j$ can appear
in the expression of $\overline{\bv}$. Since $s-2j$, $n-s+j$ and $j$ are nonnegative, 
\begin{equation}\label{eq:range-j}
\max\{0, s-n\} \le j \le [s/2]
\end{equation}
and \eqref{eq:F1} holds.

Next, we determine the coefficient $c_j^{(s,n)}$ of $y_0^{s-2j} y_1^{n-s+j} y_2^j$
in \eqref{eq:F1} for a fixed $j$. The polynomial
$\overline{\bv} = (2 y_2 \frac{\partial}{\partial y_0} - y_0 \frac{\partial}{\partial y_1})^s y_1^n$
is obtained by applying the differential operator \eqref{eq:def-CD} $s$ times to $y_1^n$.
In this process, we need to adopt $j$ times the first part
$2 y_2 \frac{\partial}{\partial y_0}$ and $s-j$ times the second part
$- y_0 \frac{\partial}{\partial y_1}$ of \eqref{eq:def-CD} to obtain
$y_0^{s-2j} y_1^{n-s+j} y_2^j$. 
Note that the operators $y_2 \frac{\partial}{\partial y_0}$ and 
$y_0 \frac{\partial}{\partial y_1}$ do not commute. 
Denote by $c_{(i_1,i_2,\ldots,i_j)}^{(s,n)}$ the coefficient of
$y_0^{s-2j} y_1^{n-s+j} y_2^j$ in the case where we adopt
the first part $2 y_2 \frac{\partial}{\partial y_0}$ in the $i_1$-th,
the $i_2$-th, $\ldots,$ the $i_j$-th turn with
\begin{equation}\label{eq:i1-ij-condition-1}
1 \le i_1 < i_2 < \cdots < i_j \le s
\end{equation}
and adopt the second part $- y_0 \frac{\partial}{\partial y_1}$ in the 
remaining $s-j$ turns.

For instance, if $j=0$ then we adopt 
$- y_0 \frac{\partial}{\partial y_1}$ in all the $s$ turns and obtain the monomial
\begin{equation*}
(-1)^s n(n-1) \cdots(n-s+1) y_0^s y_1^{n-s} = (-1)^s \frac{n!}{(n-s)!} y_0^s y_1^{n-s}.
\end{equation*}
That is, $c_0^{(s,n)} = (-1)^s n!/(n-s)!$. 
Note that \eqref{eq:range-j} implies $s \le n$ if $j = 0$.

The second simplest case is that $j=1$. In this case we adopt 
$2 y_2 \frac{\partial}{\partial y_0}$ just one time and adopt  
$- y_0 \frac{\partial}{\partial y_1}$ the remaining $s-1$ times. Suppose
we adopt $2 y_2 \frac{\partial}{\partial y_0}$ in the $i$-th turn.
Since $\frac{\partial}{\partial y_0} y_1^n = 0$,
we assume that $2 \le i \le s$.
Up to the $(i-1)$-th turn we always apply 
$- y_0 \frac{\partial}{\partial y_1}$ to $y_1^n$ and obtain the monomial
\begin{equation*}
(-1)^{i-1} n(n-1) \cdots (n-i+2) y_0^{i-1} y_1^{n-i+1}.
\end{equation*}
In the $i$-th turn we apply $2 y_2 \frac{\partial}{\partial y_0}$
to the above monomial and obtain
\begin{equation*}
(-1)^{i-1} n(n-1) \cdots(n-i+2) \cdot 2(i-1) y_0^{i-2} y_1^{n-i+1}y_2.
\end{equation*}
In the $i+1$-th turn through the $s$-th turn we apply 
$- y_0 \frac{\partial}{\partial y_1}$ to the above monomial. Then we obtain
\begin{equation*}
(-1)^{s-1} n(n-1) \cdots (n-s+2) \cdot 2(i-1) y_0^{s-2} y_1^{n-s+1}y_2.
\end{equation*}
That is,
\begin{equation*}
c_{(i)}^{(s,n)} = (-1)^{s-1} \frac{n!}{(n-s+1)!} \cdot 2(i-1).
\end{equation*}
Hence we have
\begin{equation*}
c_1^{(s,n)} = \sum_{i=2}^s c_{(i)}^{(s,n)} = (-1)^{s-1}\frac{s!\,n!}{(s-2)!\,(n-s+1)!}.
\end{equation*}

We now consider $c_{(i_1,i_2,\ldots,i_j)}^{(s,n)}$ for a general case.
We adopt $2 y_2 \frac{\partial}{\partial y_0}$ in the $i_1$-th,
the $i_2$-th, $\ldots,$ the $i_j$-th turn for $i_1, \ldots, i_j$ with
the condition \eqref{eq:i1-ij-condition-1}
and adopt $- y_0 \frac{\partial}{\partial y_1}$
in the remaining $s-j$ turns.
Just after the $(i_1-1)$-th turn we have the monomial
\begin{equation*}
(-1)^{i_1-1} n(n-1) \cdots (n-i_1+2) y_0^{i_1-1} y_1^{n-i_1+1}.
\end{equation*}
Note that $n-i_1+1 \ge 0$ by the conditions 
\eqref{eq:range-j} and \eqref{eq:i1-ij-condition-1}.

In the $i_1$-th turn we apply $2 y_2 \frac{\partial}{\partial y_0}$
to the term $y_0^{i_1-1} y_1^{n-i_1+1}$.
If $i_1 = 1$, then the operator $\frac{\partial}{\partial y_0}$ transforms the term to $0$. 
Hence we assume that $2 \le i_1$.
By the action of $2 y_2 \frac{\partial}{\partial y_0}$ we obtain
\begin{equation*}
2(i_1-1) y_0^{i_1-2} y_1^{n-i_1+1} y_2.
\end{equation*}

In the $i_1+1$-th turn through the $i_2-1$-th turn we apply 
$- y_0 \frac{\partial}{\partial y_1}$ to the term $y_0^{i_1-2} y_1^{n-i_1+1} y_2$.
Then we obtain
\begin{equation*}
(-1)^{i_2-i_1-1} (n-i_1+1)(n-i_1) \cdots (n-i_2+3) y_0^{i_2-3} y_1^{n-i_2+2} y_2.
\end{equation*}
Note that $n-i_2+2 \ge 0$ by \eqref{eq:range-j} and \eqref{eq:i1-ij-condition-1}.

In the $i_2$-th turn we apply $2 y_2 \frac{\partial}{\partial y_0}$
to the term $y_0^{i_2-3} y_1^{n-i_2+2} y_2$.
If $i_2 = 3$, then the operator $\frac{\partial}{\partial y_0}$ transforms the term to $0$.
Hence we assume that $4 \le i_2$.
By the action of $2 y_2 \frac{\partial}{\partial y_0}$ we obtain
\begin{equation*}
2(i_2-3) y_0^{i_2-4} y_1^{n-i_2+2} y_2^2.
\end{equation*}

We continue the above procedures. Then we eventually have that
\begin{equation*}
c_{(i_1,i_2,\ldots,i_j)}^{(s,n)} =
(-1)^{s-j} \frac{n!}{(n-s+j)!} \cdot
2^j (i_1-1)(i_2-3) \cdots (i_j-2j+1)
\end{equation*}
if $i_1,i_2,\ldots,i_j$ satisfy the conditions
$2 \le i_1$, $4 \le i_2$, $\ldots,$ $2j \le i_j$.
For $i_1,i_2,\ldots,i_j$ which do not satisfy these conditions, we have
$c_{(i_1,i_2,\ldots,i_j)}^{(s,n)} = 0$.

It remains to show the equation \eqref{eq:F2}.
The coefficient $c_j^{(s,n)}$ of $y_0^{s-2j} y_1^{n-s+j} y_2^j$ in
$\overline{\bv}$ is the sum of all $c_{(i_1,i_2,\ldots,i_j)}^{(s,n)}$
with respect to $i_1,i_2,\ldots,i_j$. Therefore,
\begin{align*}
c_j^{(s,n)}
&= \sum_{\substack{1 \le i_1 < i_2 < \cdots < i_j \le s\\
2 \le i_1, 4 \le i_2, \ldots, 2j \le i_j}} c_{(i_1,i_2,\ldots,i_j)}^{(s,n)}\\
&= (-1)^{s-j} \frac{n!}{(n-s+j)!} \cdot 2^j Q(s,j)\\
&=  (-1)^{s-j}\frac{s!\,n!}{(s-2j)!\,(n-s+j)!\,j!}
\end{align*}
by Lemma \ref{lem:Qsj}. The proof of Theorem
\ref{thm:f0-action-mod-C2-general} is complete.
\end{proof}

We now discuss $\overline{\mathbf{u}^0} = \mathbf{u}^0 + C_2(V(k,0)) \in R_{V(k,0)}$ 
for the singular vector
$\mathbf{u}^0$ of \eqref{eq:singular-vec}. 
Recall that $y=y_0$ and $z = y_1 y_2$.
Let
\begin{equation*}
f_0(y,z) = \frac{(-1)^{k+1}}{(k+1)!} \overline{\mathbf{u}^0} \in \mathbb{C}[y,z].
\end{equation*}

The following assertion is a consequence of Theorem \ref{thm:f0-action-mod-C2-general}.

\begin{cor}\label{eq:formula-1}
We have
\begin{equation*}
f_0(y,z) = \sum_{j=0}^{[(k+1)/2]} c_j y^{k+1-2j} z^j
\end{equation*}
with
\begin{equation*}
c_j = (-1)^j \frac{(k+1)!}{(k+1-2j)!\,(j!)^2}.
\end{equation*}
\end{cor}

We consider the action of the differential operator $D$
defined by \eqref{eq:def-D}.
For a homogeneous polynomial
\begin{equation*}
f(y,z) = \sum_{j=0}^{[n/2]} a_j y^{n-2j} z^j
\end{equation*}
of weight $n$, we have
\begin{equation*}
D\cdot f(y,z) = \sum_{j=0}^{[(n+1)/2]}
\big( -2k(n+2-2j) a_{j-1} + (k(n+j) + 2n)a_j \big) y^{n+1-2j} z^j,
\end{equation*}
where $a_{-1}$ and $a_{[n/2] + 1}$ are understood to be $0$. 
Let
\begin{equation*}
f_r(y,z) = D^r \cdot f_0(y,z)
\end{equation*}
be the image of $f_0(y,z)$
under the operator $D^r$, $r = 1,2,\ldots.$ Then $f_r(y,z)$ is a constant multiple of 
$ (W^3_1)^r \cdot \overline{\bu^0} = (W^3_1)^r \bu^0 + C_2(V(k,0)) \in \C[y,z]$  
by Proposition \ref{prop:D_and_W3p1}. 
The polynomial $f_r(y,z)$ is homogeneous of weight $k+1+r$.

We can verify that the Jacobian determinant of  $(f_0(y,z), f_1(y,z))$ is nonzero.
\begin{equation*}
\left|
\begin{array}{cc}
\partial f_0/\partial y & \partial f_0/\partial z\\
\partial f_1/\partial y & \partial f_1/\partial z
\end{array}
\right|
\ne 0.
\end{equation*}
Indeed, if $k = 2m-1$ is odd, then $\partial f_0/\partial z$ contains the term $z^{m-1}$ 
and $ \partial f_1/\partial y$ contains the term $z^m$, whereas  
$\partial f_0/\partial y$ and $\partial f_1/\partial z$ are homogeneous polynomials 
of odd weight and so every term in these two polynomials involves an odd exponent of $y$. 
If $k = 2m$ is even, then both $\partial f_0/\partial y$ and $\partial f_1/\partial z$ 
contain the term $z^m$, whereas $\partial f_0/\partial z$ and $\partial f_1/\partial y$ 
are homogeneous of odd weight. 
In either case, the Jacobian determinant is nonzero.

Since the Jacobian determinant of  $(f_0(y,z), f_1(y,z))$ is nonzero, 
$P(f_0(y,z), f_1(y,z)) = 0$ for a polynomial $P(X,Y) \in \C[X,Y]$ only if $P(X,Y) = 0$. 
That is, $f_0(y,z)$ and $f_1(y,z)$ are algebraically independent over $\C$.

We can verify the following lemma by a direct calculation.
\begin{lem}\label{lem:relation_f0f1f2}
$f_2(y,z)$ can be expressed as follows.
\begin{equation*}
f_2(y,z) = p(y,z)f_0(y,z) + q(y)f_1(y,z),
\end{equation*}
where $p(y,z) = -(k+1)(k+2)^2((k+1)y^2+kz)$ and $q(y) = (k+2)(2k+3)y$.
\end{lem}

Let $J$ be the ideal of 
$\mathbb{C}[y,z]$ defined by
\begin{equation*}
J = \mathbb{C}[y,z]f_0(y,z) + \mathbb{C}[y,z]f_1(y,z).
\end{equation*}
Then  
\begin{equation}\label{eq:in-J}
f_r(y,z) \in J\quad\text{ for }r \ge 0
\end{equation}
by the above lemma.

\begin{lem}\label{lem:J-cofinite}
$\dim \mathbb{C}[y,z]/J < \infty$.
\end{lem}

\begin{proof}
Let
$X=\{(y,z)\in \C^2\,|\, g(y,z)=0\text{ for all }g\in J\}$ be the variety 
of the zero set of $J$. 
We need to show that
$\dim X=0$,
or equivalently,
$X=\{(0,0)\}$ 
because $X$ is conic.

By (\ref{eq:in-J}), 
$D$ preserves $J$,
and hence
it defines a derivation on $\C[y,z]/J$. 
Since $J$ is a homogenous ideal, it is
also preserved by the derivation
\begin{equation*}
E = y \frac{\partial}{\partial y} + 2z \frac{\partial}{\partial z},
\end{equation*}
which corresponds to
the weight $0$ operator $L_{\mathrm{aff}}(0) = (\omega_{\mathrm{aff}})_1$ 
for the conformal vector $\omega_{\mathrm{aff}}$ of $V(k,0)$. 
Thus we have 
two elements $D, E$ in $\operatorname{Der}(\C[y,z]/J)$.
Let $D_x$ and $E_x$
be the images of
these elements
under the natural map
$\operatorname{Der}(\C[y,z]/J)\rightarrow T_{x}X$ for $x\in X$, 
where $T_x X$ is the tangent space of $X$ at $x$.

The determinant of the $2 \times 2$ matrix consisting of the coefficients of
$\partial/\partial y$ and $\partial/\partial z$ in the differential operators $D$ and
$E$ is
\begin{equation}\label{eq:det_D_Laff}
\left|
\begin{array}{cc}
(k+2)y^2 - 2kz & (3k+4)yz\\
y & 2z
\end{array}
\right|
= -k(y^2+4z)z.
\end{equation}

Suppose 
$\dim X>0$.
Then there exists a point $(y,z)\ne (0,0)$
in the smooth part of $X$.
We claim that the determinant $-k(y^2+4z)z$ can not be $0$. 
Indeed,
if $z=-y^2/4$
then 
\begin{equation*}
f_0(y, -y^2/4) = \sum_{j=0}^{[(k+1)/2]} \frac{(-1)^j}{4^j} c_j y^{k+1}
\end{equation*}
is nonzero, for $(-1)^j c_j > 0$. Also,
we  have $f_0(y,0) = y^{k+1} \ne 0$. 
We have proved that
the vectors  
$D_{(y,z)}$ and $E_{(y,z)}$ 
are linearly independent,
and hense
$\dim T_{(y,z)}X\geq 2$.
This forces that
 $\dim X\geq 2$,
which is absurd.
This completes the proof.
\end{proof}

We will study the weight $n$ subspace $J_{(n)}$ of $J$. Since $f_0(y,z)$ and
$f_1(y,z)$ are homogeneous polynomials of weight $k+1$ and $k+2$, respectively, 
we have $J_{(n)} = 0$ for $n \le k$, $J_{(k+1)} = \C f_0(y,z)$, and
\begin{equation}\label{eq:weight_subspace_J}
J_{(n)} = \C[y,z]_{(n-k-1)}f_0(y,z) + \C[y,z]_{(n-k-2)}f_1(y,z)
\end{equation}
for $n \ge k+2$. 
The next proposition plays an important role in our argument.

\begin{prop}\label{prop:kernel_pi}
We have an exact sequence of $\C[y,z]$-modules
\begin{equation}\label{eq:exact_sec_J}
0 \rightarrow \C[y,z] \stackrel{\xi_1}{\rightarrow} \C[y,z] \oplus \C[y,z] 
\stackrel{\xi_2}{\rightarrow} J \rightarrow 0,
\end{equation}
where
\begin{align*}
\xi_1 &: h(y,z) \mapsto (h(y,z)f_1(y,z), -h(y,z)f_0(y,z)), \\
\xi_2 &: (h_0(y,z), h_1(y,z)) \mapsto h_0(y,z)f_0(y,z) + h_1(y,z) f_1(y,z).
\end{align*}
That is, $a(y,z)f_0(y,z) + b(y,z)f_1(y,z) = 0$ for 
$a(y,z)$, $b(y,z) \in \C[y,z]$ if and only if $a(y,z) = h(y,z)f_1(y,z)$ and 
$b(y,z) = -h(y,z)f_0(y,z)$ for some $h(y,z) \in \C[y,z]$.
\end{prop}

\begin{proof}
Let  $\CF$  be a free $\C[y,z]$-module with basis $\{P_0, P_1\}$.
We assign the weight of $P_0$ and $P_1$ by
$\wt P_0 = k+1$ and $\wt P_1 = k+2$ so that $\CF$ is a $\Z_{\ge 0}$-graded module. 
Denote by $\CF_{(n)}$ the weight $n$ subspace of $\CF$. Then $\CF_{(n)} = 0$ for $n \le k$,
$\CF_{(k+1)} = \C P_0$ is of dimension $1$, and
\begin{equation}\label{eq:weight_subspace_F}
\CF_{(n)} = \C[y,z]_{(n-k-1)}P_0 + \C[y,z]_{(n-k-2)}P_1
\end{equation}
for $n \ge k+2$.
Let $\pi$ be a homomorphism of
$\C[y,z]$-modules from $\CF$ onto $J$ which maps $P_r$ to $f_r(y,z)$, $r = 0,1$.
\begin{equation}\label{eq:def_pi}
\pi : \CF \rightarrow J; \quad P_r \mapsto f_r(y,z),  \quad r = 0,1.
\end{equation}
Then $\CF/\Ker \pi \cong J$. We will determine the kernel $\Ker \pi$ of
$\pi$ by using the equation
\begin{equation}\label{eq:Gaussian_symbol}
[n/2]  = n-k-2 - [(n-2k-3)/2]
\end{equation}
for $n \ge 2k+3$.

Suppose $a(y,z)f_0(y,z) + b(y,z)f_1(y,z) = 0$ for some nonzero homogeneous 
$a(y,z)$, $b(y,z) \in \C[y,z]$ with $\wt a(y,z) \le k+1$ and $\wt b(y,z) \le k$. 
We choose  $a(y,z)$ and $b(y,z)$ so that
their weights are as small as possible. Then $\wt a(y,z) = m-k-1$ and
$\wt b(y,z)  = m-k-2$ for some $k+3 \le m \le 2k+2$.

Let $T = a(y,z)P_0 + b(y,z)P_1 \in \CF$. Then $\wt T = m$ and $\C[y,z]T$ is
contained in $\Ker \pi$. Take $n \ge m$. We consider the dimension of the weight
$n$ subspace. Since $\CF$ is a free $\C[y,z]$-module, it follows from
Lemma \ref{lem:basis_CA} and \eqref{eq:weight_subspace_F} that
\begin{equation}\label{eq:dim_Fn}
\begin{split}
\dim \CF_{(n)} &= [(n-k-1)/2] + [(n-k-2)/2] + 2\\
&= n-k.
\end{split}
\end{equation}

Likewise, the weight $n$ subspace of $\C[y,z]T$ is $\C[y,z]_{(n-m)}T$,
which is  $[(n-m)/2] + 1$ dimensional.
Since $\CF/\Ker \pi \cong J$ and
$\C[y,z]T \subset \Ker \pi$, we have an upper bound for the dimension of $J_{(n)}$.
\begin{equation}\label{eq:upper_bound_Jn}
\dim J_{(n)} \le  n-k-1 - [(n-m)/2].
\end{equation}

For simplicity of notation, we denote the right hand side of \eqref{eq:upper_bound_Jn}
by $Q(n,m)$.
If $n = 2k+2$, then $Q(2k+2,m) = k+1 - [(2k+2-m)/2]$ and so 
$Q(n,m) < [n/2] + 1$ for all $k+3 \le m \le 2k+2$. 
If $n = 2k+3$, then $Q(2k+3,m) = k+2 - [(2k+3-m)/2]$ and so 
$Q(n,m) = [n/2] + 1$ for $m = 2k+2$ and $Q(n,m) < [n/2] + 1$ for 
$k+3 \le m \le 2k+1$.
Now, both $Q(n,m)$ and $[n/2] + 1$ increase by $1$ when $n$ increases by $2$. 
Since $\dim \C[y,z]_{(n)} = [n/2] + 1$, it follows that
$\dim J_{(n)} < \dim \C[y,z]_{(n)}$ for infinitely many $n$.
But this is impossible by Lemma \ref{lem:J-cofinite}. 
Therefore, $a(y,z)f_0(y,z) + b(y,z)f_1(y,z) \ne 0$ 
for any nonzero homogeneous polynomials 
$a(y,z)$, $b(y,z) \in \C[y,z]$ with $\wt a(y,z) \le k+1$ and $\wt b(y,z) \le k$.

This implies that
the weight $n$ subspace $(\Ker \pi)_{(n)}$ of $\Ker \pi$ is $0$
for $n \le 2k+2$ and that $\dim J_{(n)} = n-k$ for $k+2 \le n \le 2k+2$
by \eqref{eq:weight_subspace_J}. 
In particular, $J_{(n)} = \C[y,z]_{(n)}$ for $n=2k+1$, $2k+2$.

Let $s$ be a positive integer. By Lemma \ref{lem:basis_CA}, 
$\C[y,z]_{(2s+1)} = y \C[y,z]_{(2s)}$ and
$\C[y,z]_{(2s+2)} = y \C[y,z]_{(2s+1)} + z \C[y,z]_{(2s)}$.
Hence the above paragraph in fact implies that
$J_{(n)} = \C[y,z]_{(n)}$ for all $n \ge 2k+1$.

Finally, we determine $\Ker \pi$. As shown above,
$ (\Ker \pi)_{(n)} = 0$ for $n \le 2k+2$.
Let $n \ge 2k+3$. Since $\CF/\Ker \pi \cong J$, we have
$\dim (\Ker \pi)_{(n)} = \dim \CF_{(n)} - \dim J_{(n)}$.
Moreover, $\dim J_{(n)} = [n/2] + 1$. Thus
\begin{equation*}
\dim (\Ker \pi)_{(n)} = [(n-2k-3)/2] + 1
\end{equation*}
by \eqref{eq:Gaussian_symbol} and \eqref{eq:dim_Fn}.
Now, $\C[y,z] (f_1(y,z)P_0 - f_0(y,z)P_1)$ is contained in
$\Ker \pi$ and the dimension of its weight $m$ subspace
coincides with $\dim (\Ker \pi)_{(m)}$ for all $m \ge 0$. 
Therefore, it is equal to $\Ker \pi$ and the proof is complete.
\end{proof}

In the proof of the above proposition, the weight $n$
subspace $J_{(n)}$ of $J$ is determined.
Combining the results with Lemma \ref{lem:basis_CA},
we can calculate the dimension of $\C[y,z]/J$.

\begin{lem}\label{lem:weight_subspace_J}
$(1)$ $J_{(n)} = 0$ if $n \le k$ and  $J_{(k+1)} = \C f_0(y,z)$.

$(2)$ $\dim J_{(n)} = n-k$ if $k \le n \le 2k+2$.

$(3)$ $J_{(n)} = \C[y,z]_{(n)}$ if $n \ge 2k+1$.

$(4)$ $\dim \C[y,z]/J = (k+1)(k+2)/2$.
\end{lem}

Next, we study the intersection of $J_{(n)}$ and $\CA$.

\begin{lem}\label{lem:J_cap_CA}
$(1)$ $J_{(n)} \cap \CA = 0$ if $n \le k$ and $J_{(k+1)} \cap \CA = \C f_0(y,z)$.

$(2)$ $\dim J_{(n)} \cap \CA = n-k-1$ if $k+2 \le n \le 2k+2$.

$(3)$ $J_{(n)} \cap \CA = \CA_{(n)}$ if $n \ge 2k+1$.

$(4)$ $\dim \CA/(J \cap \CA) = k(k+1)/2$.
\end{lem}

\begin{proof}
The assertions (1) and (3) are clear from Lemma \ref{lem:weight_subspace_J}.
Take $n$ such that $k+2 \le n \le 2k+2$ and multiply $f_0(y,z)$ by $y^{n-k-1}$.
Then we obtain an element
\begin{equation*}
y^{n-k-1} f_0(y,z) = y^n - (k+1)k y^{n-2}z^2 + (\textrm{terms of the form }
y^{n-2j}z^j, j \ge 2)
\end{equation*}
of $J_{(n)}$, which is not contained in $\CA$ by Lemma \ref{lem:basis_CA}. 
Since $\dim \CA_{(n)} = \dim \C[y,z]_{(n)} - 1$, 
we have  $\dim J_{(n)} \cap \CA = \dim J_{(n)}  - 1$
and the assertion (2) holds. Now, we can calculate the dimension of
$\CA/(J \cap \CA)$ and obtain the assertion (4).
\end{proof}

We consider some ideals of $\CA$ generated by $f_r(y,z)$, $r = 0, 1, 2, \cdots$.
Let $I_s$ be the ideal of $\mathcal{A}$ generated by
$f_r(y,z)$, $0 \le r \le s-1$ and
$I$ the ideal of $\CA$ generated by $f_r(y,z)$, $r \ge 0$.
We will calculate the dimension of weight $n$ subspaces $I_{s(n)}$
and $I_{(n)}$ of these ideals.
First, we study $I_2 = \CA f_0(y,z) + \CA f_1(y,z)$.

\begin{lem}\label{lem:weight_subspace_I2}
$(1)$ $ I_{2(n)} = 0$ if $n \le k$ and $\dim I_{2(n)} = 1$ if $n= k+1, k+2$.

$(2)$  $\dim I_{2(n)} = n-k-2$ if $k+3 \le n \le 2k+2$.

$(3)$  $\dim I_{2(2k+3)} = k$.

$(4)$ $I_{2(n)} = \CA_{(n)}$ if $n \ge 2k+4$.

$(5)$ $\dim \CA/I_2 = (k+1)(k+2)/2$.
\end{lem}

\begin{proof}
The weight $n$ subspace of $I_2$ is 
$I_{2(n)} = \CA_{(n-k-1)} f_0(y,z) + \CA_{(n-k-2)} f_1(y,z)$. 
Since $\CA_{(0)} = \C$ and $\CA_{(1)} = 0$,
we have  $I_{2(k+1)} = \C f_0(y,z)$ and $I_{2(k+2)} = \C f_1(y,z)$.
Thus the assertion (1) holds. 
The assertions (2), (3) and (4) follow from Lemma \ref{lem:basis_CA} 
and Proposition \ref{prop:kernel_pi}. 
Now, we can calculate the dimension of $\CA/I_2$ and obtain the assertion (5).
\end{proof}

Next, we study $I_3 = I_2 + \CA f_2(y,z)$.

\begin{lem}\label{lem:weight_subspace_I3}
$(1)$ $I_{3(n)} = I_{2(n)}$ if $n \le k+2$, $n = k+4$, or $n \ge 2k+4$.

$(2)$ $\dim I_{3(n)} = \dim I_{2(n)} + 1$ if $n = k+3$ or $k+5 \le n \le 2k+3$.

$(3)$ $I_{3(n)} = \CA_{(n)}$ if $n \ge 2k+1$.

$(4)$ $\dim \CA/I_3 = k(k+1)/2 + 1$.
\end{lem}

\begin{proof}
Since $f_2(y,z)$ is a homogeneous polynomial of weight $k+3$, we have 
$I_{3(n)} = I_{2(n)}$ for $n \le k+2$.
Let $n \ge k+3$. Then the weight $n$ subspace of $I_3$ is
\begin{equation}\label{eq:weight_subspace_I3}
I_{3(n)} =  \CA_{(n-k-1)} f_0(y,z) + \CA_{(n-k-2)} f_1(y,z) + \CA_{(n-k-3)} f_2(y,z).
\end{equation}

If $n=k+4$, then $I_{3(n)} = I_{2(n)}$ since $\CA_{(1)} = 0$.
If $n \ge 2k+4$, then $I_{3(n)} = I_{2(n)}$ by
Lemma \ref{lem:weight_subspace_I2} (4).
Hence the assertion (1) holds.

There is no $\alpha \in \C$ such that $f_2(y,z) = \alpha (y^2-2kz) f_0(y,z)$ 
by Lemma \ref{lem:relation_f0f1f2} and Proposition \ref{prop:kernel_pi}. 
Since $\CA_{(0)} = \C$,  $\CA_{(1)} = 0$ and $\CA_{(2)} = \C (y^2-2kz)$, 
we see that $I_{3(k+3)} = I_{2(k+3)} \oplus \C f_2(y,z)$.

It remains to discuss the case $k+5 \le n \le 2k+3$.
In this case we study a necessary and sufficient condition for
$c(y,z) \in \CA_{(n-k-3)} $ such that the equation
\begin{equation}\label{eq:condition_on_cyz}
c(y,z) f_2(y,z) = a(y,z) f_0(y,z) + b(y,z) f_1(y,z)
\end{equation}
holds for some $a(y,z) \in \CA_{(n-k-1)} $, $b(y,z) \in \CA_{(n-k-2)}$.
We deal with two cases $k+5 \le n \le 2k+2$ and $n = 2k+3$ separately.
Suppose $k+5 \le n \le 2k+2$. Then Lemma \ref{lem:relation_f0f1f2} and
Proposition \ref{prop:kernel_pi} imply that \eqref{eq:condition_on_cyz}
is equivalent to the system of equations
\begin{equation}\label{eq:condition_on_cyz_2}
\begin{split}
a(y,z) - p(y,z) c(y,z) &= 0,\\
b(y,z) - q(y) c(y,z) &= 0.
\end{split}
\end{equation}

Set $m=n-k-3$ and recall the basis of $\CA_{(m)}$ in Lemma \ref{lem:basis_CA}. 
If $c(y,z) =y^m - m k y^{m-2} z \in \CA_{(m)}$,
then $y c(y,z) = y^{m+1} - m k y^{m-1} z$ is not contained in $\CA$.
Thus $b(y,z) - q(y) c(y,z) \ne 0$ for any $b(y,z) \in \CA_{(n-k-2)} $.
If $c(y,z) = y^{m-2j} z^j$ for some $2 \le j \le [m/2]$, then
the exponent of $z$ in each term appears in $p(y,z) c(y,z)$ and
$q(y) c(y,z)$ is at least $2$ and so $p(y,z) c(y,z)$ and
$q(y) c(y,z)$ are contained in $\CA$. Hence there exist
$a(y,z) \in \CA_{(n-k-1)} $ and $b(y,z) \in \CA_{(n-k-2)} $ which
satisfy \eqref{eq:condition_on_cyz_2}.
Thus $\dim I_{3(n)} = \dim I_{2(n)} + 1$ for $k+5 \le n \le 2k+2$.

Finally, we consider the case $n=2k+3$. In this case,
Proposition \ref{prop:kernel_pi} implies that 
\eqref{eq:condition_on_cyz} holds if and only if
\begin{equation}\label{eq:condition_on_cyz_4}
\begin{split}
a(y,z) - p(y,z) c(y,z) &= \beta f_1(y,z),\\
b(y,z) - q(y) c(y,z) &= -\beta f_0(y,z)
\end{split}
\end{equation}
for some $\beta \in \C$. 
If $c(y,z) =y^k - k^2 y^{k-2} z \in \CA_{(k)}$, then
$y c(y,z) \not\in \CA$ and so there is no $b(y,z) \in \CA_{(k+1)}$
which satisfies \eqref{eq:condition_on_cyz_4}.
If $c(y,z) = y^{k-2j} z^j$ for some $2 \le j \le [k/2]$, then as before
we see that there are $a(y,z) \in \CA_{(k+2)} $ and $b(y,z) \in \CA_{(k+1)} $
which satisfy \eqref{eq:condition_on_cyz_4}.

We have shown that
\begin{equation*}
I_{3(n)} = I_{2(n)} \oplus \C \big( y^{n-k-3} - (n-k-3) k y^{n-k-5} z \big) f_2(y,z)
\end{equation*}
for $k+5 \le n \le 2k+3$. This completes the proof of the assertion (2).
Now, the assertions (3) and (4) are clear from Lemma \ref{lem:weight_subspace_I2}.
\end{proof}

We are now in a position to discuss $I_4 = I_3 + \CA f_3(y,z)$.

\begin{lem}\label{lem:weight_subspace_I4}
$(1)$ $I_4 = I_3 \oplus \C f_3(y,z)$.

$(2)$ $\dim \CA/I_4 = k(k+1)/2$.
\end{lem}

\begin{proof}
We want to show that $I_3$ does not contain $f_3(y,z)$.
By Lemma \ref{lem:relation_f0f1f2}, the action of the differential operator
$D$ on $f_2(y,z)$ gives
\begin{align*}
f_3(y,z)
&= (D \cdot p(y,z)) f_0(y,z) + \big( p(y,z) + D \cdot q(y) \big) f_1(y,z) + q(y) f_2(y,z)\\
&= \big( D \cdot p(y,z) + p(y,z) q(y) \big) f_0(y,z)
+ \big( p(y,z) + D \cdot q(y) + q(y)^2 \big) f_1(y,z).
\end{align*}
Moreover, $D \cdot p(y,z) + p(y,z) q(y)$ is
\begin{equation*}
(k+1)(k+2)^2 \big( (k+1)(k+2)(2k+5) y^3 - 2k(k^2+3k+3)yz \big)
\end{equation*}
and $p(y,z) + D \cdot q(y) + q(y)^2$ is
\begin{equation*}
(k+2) \big( (k+2) (3k^2+12k+11) y^2 - k (k^2+7k+8) z \big),
\end{equation*}
neither of which is contained in $\CA$ by  Lemma \ref{lem:basis_CA}. 
On the other hand, we know that the weight $k+4$ subspace of $I_3$ is
\begin{equation*}
I_{3(k+4)} = \C (y^3 - 3kyz) f_0(y,z) + \C (y^2 - 2kz)  f_1(y,z),
\end{equation*}
since $\CA_{(1)} = 0$. Thus $I_{3(k+4)}$ does not contain  $f_3(y,z)$ by
Proposition \ref{prop:kernel_pi}.

This, together with Lemma \ref{lem:weight_subspace_I3} (4) implies that
$\dim \CA/I_4 \le k(k+1)/2$. Since $I_4 \subset J \cap \CA$,
It follows from Lemma \ref{lem:J_cap_CA} (4) that
$I_4 = J \cap \CA$ and we obtain the assertions.
\end{proof}

As mentioned in the proof of the above lemma, we have $I_4 = J \cap \CA$.
Since $I_4 \subset I \subset J \cap \CA$ by the definition of $I$, the
following theorem holds.

\begin{thm}\label{thm:CA_I}
$(1)$ $\dim \CA/I = k(k+1)/2$.

$(2)$ $I =  I_4$ and $J \cap \CA = I$.
\end{thm}

\section{Embedding of $R_{\CW}$ in $R_{L(k,0)}$}\label{sec:embedding_C2CW}

In this section we study Zhu's $C_2$-algebra $R_{\CW}$ of $\CW$. 
Let $\CU = V(k,0)(0)$, that is, $\CU = \{ v \in V(k,0)\,|\, h(0) v = 0\}$. 
Then $\CN \subset \CU \subset V(k,0)$. 
As before we denote the image of a subset $S$ of $V(k,0)$ in 
$R_{V(k,0)}$ by $\overline{S} = S + C_2(V(k,0))$. 
Note that $\overline{\CU} = \C[y,z]$. 

For $u, v \in V(k,0)$, we have $h(0) u_{-2} v = (h(0)u)_{-2} v + u_{-2} h(0)v$ 
by \eqref{eq:VOA-formula-1}. 
Thus $C_2(V(k,0))$ is invariant under $h(0)$ and $h(0)$ induces an action on 
$R_{V(k,0)}$ by $h(0)\cdot \overline{v} = \overline{h(0)v}$. 
Moreover, $h(0) u_{-1} v = (h(0)u)_{-1} v + u_{-1} h(0)v$ 
implies that $h(0)$ acts on $R_{V(k,0)} = \C[y_0,y_1,y_2]$ as a derivation. 
Actually, $h(0) \cdot y_0^p y_1^q y_2^r = 2(q-r) y_0^p y_1^q y_2^r$ 
by \eqref{eq:def-y0y1y2}. 

Since $L(k,0) = V(k,0)/\CJ$, we see that 
$R_{L(k,0)} = L(k,0)/C_2(L(k,0))$ is isomorphic to 
$V(k,0)/ (C_2(V(k,0)) + \CJ) \cong R_{V(k,0)}/\overline{\CJ}$. 
We identify $R_{L(k,0)}$ with $R_{V(k,0)}/\overline{\CJ}$ for simplicity 
and write $\overline{v} + \overline{\CJ}$, $v \in V(k,0)$ for an 
element of  $R_{L(k,0)}$.

A basis of $L(k,0)$ in known \cite{FKLMM, MP} (see \cite[Theorem 1.1]{Feigin} also.)  
Using the above convention, we see that 
\begin{equation}\label{eq:RL-basis}
y_0^p y_1^q y_2^r + \overline{\CJ}, \quad r \le k, p+r \le k, p+q \le k 
\end{equation}
form a basis of $R_{L(k,0)}$. 
Since $\CJ$ is an ideal of $V(k,0)$, its image $\overline{\CJ}$ in  $R_{V(k,0)}$ 
is invariant under $h(0)$ and so $h(0)$ acts on $R_{L(k,0)}$ by 
$h(0) \cdot (\overline{v} + \overline{\CJ}) = h(0) \cdot \overline{v} + \overline{\CJ}$. 
With respect to this action,  
the element $y_0^p y_1^q y_2^r + \overline{\CJ}$ is an eigenvector for 
$h(0)$ with eigenvalue $2(q-r)$. Hence  
$y^p z^q + \overline{\CJ}$, $q \le k$, $p+q \le k$ form a basis of 
$R_{L(k,0)}^{\h} = \{ x \in R_{L(k,0)} \,|\, h(0) \cdot x = 0\}$. 
We also note that 
\begin{equation}\label{eq:RLH}
R_{L(k,0)}^{\h} = (\overline{\CU} + \overline{\CJ})/\overline{\CJ} 
\cong \overline{\CU}/(\overline{\CU} \cap \overline{\CJ}).
\end{equation}

The element $y^p z^q + \overline{\CJ}$ is of weight $p + 2q$. Thus 
$y^p z^q + \overline{\CJ}$, $q \le k$, $p+q \le k$, $p + 2q = n$ 
form a basis of the weight $n$ subspace $(R_{L(k,0)}^{\h})_{(n)}$ of $R_{L(k,0)}^{\h}$. 
In particular, 
\begin{equation}\label{eq:dim-RLHN}
\dim (R_{L(k,0)}^{\h})_{(n)} = 
\begin{cases}
[n/2] + 1 & \text{if } n \le k,\\
[n/2] + 1 -n + k & \text{if } k+1 \le n \le 2k,\\
0 &  \text{if }  n \ge 2k+1.
\end{cases}
\end{equation}
Note that $[n/2] + 1 -n + k$ is equal to $k - [n/2] + 1$ or $k - [n/2] $ 
according as $n$ is even or odd for $k+1 \le n \le 2k$. 

Recall that $J = \C[y,z]f_0(y,z) + \C[y,z]f_1(y,z)$.

\begin{lem}\label{lem:J_in_CU}
$J = \overline{\CU} \cap \overline{\CJ}$.
\end{lem}

\begin{proof}
By the definition of $\CJ$, it contains $\bu^0$ and $W^3_1 \bu^0$. 
Hence $\overline{\CJ}$ contains $f_0(y,z)$ and $f_1(y,z)$. 
If $v \in \CJ$, then $h(-1)^p e(-1)^q f(-1)^q v \in \CJ$ and 
$y^p z^q \overline{v} \in \overline{\CJ}$ for $p, q \ge 0$. 
Hence $J \subset  \overline{\CJ}$. 
We compare the dimension of the weight $n$ subspaces of 
$J$ and $\overline{\CU} \cap \overline{\CJ}$. 
It follows from \eqref{eq:RLH} that $\dim (\overline{\CU} \cap \overline{\CJ})_{(n)}$ 
is the difference of $\dim \overline{\CU}_{(n)}$ and $\dim (R_{L(k,0)}^{\h})_{(n)}$. 
Since $\dim \overline{\CU}_{(n)} = [n/2] + 1$ by Lemma \ref{lem:basis_CA}, 
we see that $\dim J_{(n)}$ coincides with $\dim (\overline{\CU} \cap \overline{\CJ})_{(n)}$ 
for all $n \ge 0$ by \eqref{eq:dim-RLHN} and Lemma \ref{lem:weight_subspace_J}. 
Thus the assertion holds.
\end{proof}

\begin{lem}\label{lem:I_in_CN}
$I = \overline{\CN} \cap \overline{\CJ} = \overline{\CI}$.
\end{lem}

\begin{proof}
Recall that $\overline{\CN} = \CA$ by the definition of $\CA$ 
and $J \cap \CA = I$ by Theorem \ref{thm:CA_I}. 
Then Lemma \ref{lem:J_in_CU} implies $I = \overline{\CN} \cap \overline{\CJ}$. 
We also have $\overline{\CI} \subset \overline{\CN} \cap \overline{\CJ}$, 
for $\CI = \CN \cap \CJ$ by \cite[Lemma 3.1]{DLY}. 
The maximal ideal $\CI$ of $\CN$ contains $(W^3_1)^r \bu^0$ and so 
$\overline{\CI}$ contains $f_r(y,z)$, $r \ge 0$. 
Since $I$ is the ideal of $\overline{\CN}$ generated by $f_r(y,z)$, $0 \le r \le 3$, 
we have  $I \subset \overline{\CI}$. 
This proves the lemma.
\end{proof}

\begin{thm}\label{thm:embedding_RW_in_RL}
$(1)$ $\dim R_{\CW} = k(k+1)/2$. In particular, $\CW$ is $C_2$-cofinite.

$(2)$ $C_2(\CW) = \CW \cap C_2(L(k,0))$ and $R_{\CW} \hookrightarrow R_{L(k,0)}$.
\end{thm}

\begin{proof}
The map $\psi : R_{\CN} \rightarrow R_{V(k,0)}$; $v + C_2(\CN) \mapsto v + C_2(V(k,0))$ 
is injective by Theorem \ref{thm:embedding_RN_in_RV}. 
Since $\psi (R_{\CN}) = \overline{\CN}$ and  $\psi (\CI + C_2(\CN)) = \overline{\CI}$, 
it follows that $\overline{\CN}/ \overline{\CI} \cong 
\CN/(\CI + C_2(\CN)) \cong R_{\CW}$. 
Then Theorem \ref{thm:CA_I} and Lemma \ref{lem:I_in_CN} imply the first assertion. 

Moreover,
$\overline{\CN}/ \overline{\CI} \cong (\overline{\CN} + \overline{\CJ})/ \overline{\CJ} 
\subset (\overline{\CU} + \overline{\CJ})/ \overline{\CJ} 
= R_{L(k,0)}^{\h} \subset R_{L(k,0)}$ by Lemma \ref{lem:I_in_CN} and \eqref{eq:RLH}. 
Thus $R_{\CW} \hookrightarrow R_{L(k,0)}$ and the second assertion holds.
\end{proof}

We know $\CN \cap \CJ = \CI$ and $\CN \cap C_2(V(k,0)) = C_2(\CN)$. 
Moreover, Lemma \ref{lem:I_in_CN} implies that 
\begin{equation*}
(\CN +  C_2(V(k,0))) \cap (\CJ + C_2(V(k,0))) = \CI + C_2(V(k,0)).
\end{equation*}
Taking the intersection with $\CN$, we obtain 
$\CN \cap (\CJ + C_2(V(k,0))) = \CI + C_2(\CN)$. 

Recall that $\CU = V(k,0)(0) = M(k,0) \otimes \CN$, 
where $M(k,0) = M_{\widehat{\h}}(k,0)$ is the Heisenberg vertex operator algebra. 
Thus Zhu's $C_2$-algebra of  $\CU$ is 
$R_\CU = R_{M(k,0)} \otimes R_\CN$, which is isomorphic to $\C[y] \otimes \CA$. 
Hence $\CU \cap C_2(V(k,0)) \ne C_2(\CU)$, for $\overline{\CU} = \C[y,z]$.

By Theorem \ref{thm:CA_I} and Lemma \ref{lem:I_in_CN}, we have $I_4 = \overline{\CI}$. 
This implies that the four vectors 
$\bu^0$, $W^3_1 \bu^0$, $(W^3_1)^2 \bu^0$ and $(W^3_1)^3 \bu^0$ 
carry sufficient information about the maximal ideal $\CI$ modulo $C_2(\CN)$ 
in a sense. Recall that the three null fields $\bv^0$, $W^3_1 \bv^0$ and 
$(W^3_1)^2 \bv^0$ are sufficient for the proof of Theorem \ref{thm:embedding_RN_in_RV}. 
The information coming from $(W^3_1)^r \bv^0$, $r = 0,1,2$ and that from 
$(W^3_1)^r \bu^0$, $r = 0,1,2,3$ are independent of each other 
modulo $C_2(\CN)$.

\section{Zhu's algebra $A(\CW)$}\label{sec:Zhu_alg_CW}

In this section we study Zhu's algebra $A(\CW)$ of $\CW$. 
We use the filtrations discussed in Section \ref{sec:basic_properties_C2}.
Since $\CW$ is the commutant of the Heisenberg vertex operator algebra 
$M_{\widehat{\h}}(k,0)$ in $L(k,0)$, 
the action of the operator $(\omega_{\mraff})_1$ on $\CW$ agrees with that of 
the operator $W^2_1$, where $\omega_{\mraff}$ denotes the conformal vector 
of $L(k,0)$. 
Thus the weight $n$ subspace $\CW_{(n)}$ of $\CW$ is 
contained in the weight $n$ subspace $L(k,0)_{(n)}$ of $L(k,0)$. 
Since $\CW \subset L(k,0)$, we have $O(\CW) \subset O(L(k,0))$ and 
$A(\CW) \rightarrow A(L(k,0))$; $a + O(\CW) \mapsto a + O(L(k,0))$ 
is an algebra homomorphism. 
This homomorphism maps $F_p A(\CW)$ into $F_p A(L(k,0))$, and so 
it induces a homomorphism $\gr A(\CW) \rightarrow \gr A(L(k,0))$ of 
$\Z$-graded Poisson algebras. 

Likewise, the assignment $a + C_2(\CW)_{(p)} \mapsto a + C_2(L(k,0))_{(p)}$ 
for $a \in \CW_{(p)}$ 
induces a homomorphism $R_\CW \rightarrow R_{L(k,0)}$ of 
$\Z$-graded Poisson algebras. 
Thus we have a commutative diagram
\begin{equation}\label{eq:CD}
\begin{CD}
R_{\CW} @>>> R_{L(k,0)}\\
@VVV            @VVV\\
\gr A(\CW) @>>> \gr A(L(k,0))
\end{CD}
\end{equation}
of $\Z$-graded Poisson algebras by Proposition \ref{prop:gr_AV_onto_gr_RV}. 
It is known that $\dim R_{L(k,0)} = (k+1)(k+2)(2k+3)/6$ and the homomorphism  
$R_{L(k,0)} \rightarrow \gr A(L(k,0))$ is an isomorphism 
(\cite{Feigin}, \cite[Section IV.2]{GG}). 
Since the homomorphism $R_{\CW} \rightarrow R_{L(k,0)}$ is injective by 
Theorem \ref{thm:embedding_RW_in_RL}, the commutative diagram \eqref{eq:CD} 
implies that the surjective homomorphism $R_{\CW} \rightarrow \gr A(\CW)$ 
is in fact an isomorphism and that the homomorphism 
$\gr A(\CW) \rightarrow \gr A(L(k,0))$ is injective. 
Hence $\dim R_{\CW} = \dim A(\CW)$ and  
$A(\CW) \rightarrow A(L(k,0))$; $a + O(\CW) \mapsto a + O(L(k,0))$ is injective. 
Thus the following theorem holds.
\begin{thm}\label{thm:dim_Zhu_CW}
$(1)$ $\dim A(\CW)  = k(k+1)/2$.

$(2)$ $O(\CW) = \CW \cap O(L(k,0))$ and $A(\CW) \hookrightarrow A(L(k,0))$.  
\end{thm}

Next, we show that $A(\CW)$ is semisimple. Recall that $A(\CW)$ is commutative 
\cite[Lemma 2.6]{DLY}. 
Since $A(\CW) \rightarrow A(L(k,0))$; $a + O(\CW) \mapsto a + O(L(k,0))$ is injective, 
it is sufficient to show that $A(\CW)$ acts semisimply on $A(L(k,0))$. 
By \cite[Theorem 3.1.3]{FZ}, $A(L(k,0))$ is a semisimple algebra and 
as an $sl_2$-module it is a direct sum of finite dimensional irreducible 
modules $U^i$ with highest weight $i$,  $0 \le i \le k$. 
Moreover, \cite[Proposition 4.5]{DLY} implies that $A(\CW)$ 
acts semisimply on $U^i$. 
Thus $A(\CW)$ acts semisimply on $A(L(k,0))$ as desired.

In \cite[Section 4]{DLY}, $k(k+1)/2$ irreducible $\CW$-modules $M^{i,j}$, 
$0 \le i \le k$, $0 \le j \le i-1$ were constructed. 
We now show that these irreducible $\CW$-modules are inequivalent. 
Let $\C v$ be an irreducible $A(\CW)$-module. Then the induced module 
$A(L(k,0)) \otimes_{A(\CW)} \C v$ is a finite dimensional 
$A(L(k,0))$-module and it decomposes into a direct sum of $U^i$'s.  
It follows that $\C v$ appears as a direct summand of the 
restriction $\Res^{A(L(k,0))}_{A(\CW)} U^i$ of some $U^i$ to 
an $A(\CW)$-module. 
Such an irreducible $A(\CW)$-module must be isomorphic to 
the top level $\C v^{i,j}$ of $M^{i,j}$ studied in \cite[Section 4]{DLY} for some $j$. 
That is, $\C v \cong \C v^{i,j}$ as $A(\CW)$-modules for some $i, j$. 
Since $A(\CW)$ is commutative, semisimple and of dimension $k(k+1)/2$, 
it has exactly  $k(k+1)/2$ inequivalent irreducible modules. 
Hence the $k(k+1)/2$ irreducible $A(\CW)$-modules $\C v^{i,j}$, 
$0 \le i \le k$, $0 \le j \le i-1$ are all inequivalent.
Therefore, the following theorem is proved.

\begin{thm}\label{thm:Zhu_alg_CW}
$(1)$ $A(\CW)$ is semisimple.

$(2)$ The $k(k+1)/2$ irreducible $\CW$-modules  $M^{i,j}$, 
$0 \le i \le k$, $0 \le j \le i-1$ constructed in \cite{DLY} 
form a complete set of isomorphism classes of irreducible $\CW$-modules.
\end{thm}

We make some remarks about Conjecture 4.6 of \cite{DLY}. 
Recall that $W^2$ and $\CW$ are denoted by $\omega$ and $K_0 = M^{0,0}$, 
respectively in \cite{DLY} (see \cite[Theorem 4.1]{DLWY} also.) 
The above theorem gives an affirmative answer to part of the conjecture. 
The $C_2$-cofiniteness of $\CW$ is obtained in 
Theorem \ref{thm:embedding_RW_in_RL}. 

As to Zhu's algebra $A(\CW)$, it is known to be commutative and generated by 
four elements $[W^s] = W^s + O(\CW)$, $s = 2,3,4,5$ \cite[Lemma 2.6]{DLY}. 
In the case $k = 16$, we can verify by \cite[Proposition 4.5]{DLY} that 
the two vectors $v^{i,j}$, $(i,j) = (2,1)$ and $(8,0)$ have 
common eigenvalues for $o(W^s)$, $s = 2,3,5$, but not for  $o(W^4)$. 
In the case $k = 100$, the two vectors  $v^{i,j}$, $(i,j) = (12,1)$ and $(12,11)$ have 
common eigenvalues for $o(W^s)$, $s = 2,3,4$, but not for  $o(W^5)$. 
These examples imply that $[W^4]$ and $[W^5]$ are necessary to generate $A(\CW)$. 
Thus the latter half of Conjecture 4.6 (2) of \cite{DLY} is incorrect.

\section{Projectivity of $\CW$}\label{sec:projectivity_CW}

In this section we show that the parafermion vertex operator algebra 
$\CW$ is projective, that is, 
any $\CW$-module extension of the $\CW$-module $\CW$ itself splits. 
For this purpose, we need the action of  $o(W^2)$ and $o(W^3)$ 
on the top level  $\C v^{i,j}$ of the $k(k+1)/2$ irreducible $\CW$-modules 
$M^{i,j}$ \cite[Proposition 4.5]{DLY}. 
The argument here is a standard one.

Let 
\begin{equation}\label{eq:extention}
0 \rightarrow M \rightarrow N \stackrel{f}{\rightarrow} \CW 
\rightarrow 0
\end{equation}
be a short exact sequence of $\CW$-modules with $M$ an irreducible module. 
We want to show that the short exact sequence splits.  
By Theorem \ref{thm:Zhu_alg_CW}, $M \cong M^{i,j}$ for some $i,j$. 
Then the top level of $M$ is one dimensional and its weight is a nonnegative integer.
Thus the top level of $N$ is the weight $0$ subspace $N_{(0)}$.

Suppose the top level of $M$ is of weight $0$. Then $M$ is isomorphic to 
$M^{0,0} = \CW$ by \cite[Proposition 4.5]{DLY} and $\dim N_{(0)} = 2$. 
Since Zhu's algebra $A(\CW)$ of $\CW$ is commutative and semisimple 
by Theorem \ref{thm:Zhu_alg_CW}, 
$N_{(0)}$ is a direct sum of one dimensional  $A(\CW)$ -modules, say 
$N_{(0)} = \C v^1 \oplus \C v^2$. Let $N^r$ be the submodule of $N$ 
generated by $v^r$, $r = 1,2$. Then 
$N^r = \spn_\C \{ a_n v^r\,|\, a \in \CW, n \in \Z \}$. 
Since $\wt v^r = 0$, $a_n v^r$ is contained in $N_{(0)}$ only if the operator 
$a_n$ is of weight $0$, that is $a_n = o(a)$. Hence $N^r \cap N_{(0)} = \C v^r$. 
This implies that $N = N^1 \oplus N^2$. 
Thus the extension \eqref{eq:extention} splits.

Next, suppose the weight of the top level of $M$ is greater than $1$. 
In this case the weight $1$ subspace $N_{(1)}$ of $N$ is trivial, for 
$N/M \cong \CW$ and $\CW_{(1)} = 0$. 
Take a nonzero element $v \in N_{(0)}$ such that $f(v) = \1$. 
Then $L(-1)v = 0$ where $L(-1) = W^2_0$, for $L(-1)v$ is of weight $1$. 
Hence $v$ is a vacuum-like vector by \cite[Corollary 4.7.6]{LL}, 
\cite[Proposition 3.3]{Li} and 
the map $g: a \mapsto a_{-1}v$ is a homomorphism of $\CW$-modules from 
$\CW$ to $N$ by \cite[Proposition 4.7.7]{LL}. 
Thus the extension \eqref{eq:extention} splits. 

Finally, suppose the top level of $M$ is of weight $1$. 
Then $\dim N_{(0)} =  \dim N_{(1)} = 1$. 
Take a nonzero element $v \in N_{(0)}$ such that $f(v) = \1$. Then
\begin{equation}\label{eq:top-level-condition}
W^s_{s-1+n} v = 0 \quad \text{for } n \ge 0, s = 2,3,4,5.
\end{equation} 

We want to study the action of  $o(W^3) = W^3_2$ on the four 
elements $W^2_0 v$, $W^3_1 v$, $W^4_2 v$ and $W^5_3 v$ of weight $1$. 
For this, we use the expression of $W^3_n W^s$ as a linear combination of 
elements of the form \eqref{eq:normal-form} in \cite[Appendix B]{DLY} 
and the properties of the Virasoro operators $L(-1) = W^2_0$ and $L(0) = W^2_1$. 
We can calculate the action of  $o(W^3)$ on those four elements by the 
formulas \eqref{eq:VOA-formula-1}, \eqref{eq:VOA-formula-2}, 
\eqref{eq:Lm1-derivative} and 
the conditions \eqref{eq:top-level-condition} on the vector $v$.

For instance, the action of $o(W^3)$ on $W^4_2 v$ is 
$o(W^3) W^4_2 v = W^3_2 W^4_2 v = [W^3_2, W^4_2] v$, since $W^3_2 v = 0$. 
It follows from \eqref{eq:VOA-formula-1} that 
\begin{equation*}
[W^3_2, W^4_2] = (W^3_0 W^4)_4 + 2(W^3_1 W^4)_3 + (W^3_2 W^4)_2.
\end{equation*}
By \cite[Appendix B]{DLY} , $W^3_0 W^4$ is a linear combination of 
$W^2_{-2}W^3$, $W^2_{-1} W^3_{-2}\1$, $W^3_{-4}\1$ and $W^5_{-2}\1$. Since 
\begin{equation*}
(W^2_{-2}W^3)_4 = \sum_{i \ge 0} (-1)^i \binom{-2}{i}
\big( W^2_{-2-i} W^3_{4+i} -  W^3_{2-i} W^2_i \big), 
\end{equation*}
we have $(W^2_{-2}W^3)_4 v = - W^3_2 W^2_0 v = -2 W^3_1 v$ by 
\eqref{eq:Lm1-derivative} and the conditions \eqref{eq:top-level-condition}. 
Likewise, $(W^2_{-1} W^3_{-2}\1)_4 v = -6W^3_1 v$, 
$(W^3_{-4}\1)_4 v = -4W^3_1 v$ and $(W^5_{-2}\1)_4 v = -4W^5_3 v$. 
In this way we can express $o(W^3) W^4_2 v$ as a linear combination of the four 
elements $W^s_{s-2} v$, $s = 2,3,4,5$.
Note that the four elements $W^s_{s-2} v$, $s = 2,3,4,5$ are not linearly independent, 
for $\dim N_{(1)} = 1$. 

In fact, we have $o(W^3)W^r_{r-2} v = \sum_{s = 2}^5 a_{rs} W^s_{s-2} v$ with
\begin{align*}
a_{23} &= 2,\\
a_{32} &= 54 k^3 (k -2) (k + 2) (3 k + 4)/(16 k + 17),\\
a_{34} &= 18 k (2 k + 3)/(16 k + 17),\\
a_{43} &= 32 k^2 (k - 3) (2 k + 1) (2 k + 3) (2 k + 7)/(64 k + 107),\\
a_{45} &= -24 k (3 k + 4) (16 k + 17)/(5 (64 k + 107)),\\
a_{52} &= 120 k^4 (k + 2) (2 k + 1) (2 k + 3) (3 k + 4) (8 k^2 +  5 k + 5)/(16 k + 17),\\
a_{54} &= -15 k^2 (208 k^3 + 649 k^2 + 580 k + 120)/(2 (16 k + 17)),
\end{align*}
and the other $a_{rs}$'s are $0$. 
The $4 \times 4$ matrix $(a_{rs})_{2 \le r,s \le 5}$ has four eigenvalues
\begin{equation}\label{eq:eigenvalues}
\pm 6 k^2 (k + 2), \qquad \pm 6 k^2 (3 k + 4).
\end{equation}

The four elements $W^s_{s-2} v$, $s = 2,3,4,5$ are of weight $1$, 
so that $o(W^2)$ acts as $1$ on these elements. 
We consider the following system of equations.
\begin{align*}
\frac{1}{2k(k+2)}\Big( k(i-2j) - (i-2j)^2 + 2k(i-j+1)j \Big) &= 1,\\
k^2 (i-2j) - 3k(i-2j)^2 + 2(i-2j)^3 - 6k(i-2j)(i-j+1)j &= \lambda,
\end{align*}
where the left hand side is the eigenvalue of the action of $o(W^2)$ and 
$o(W^3)$, respectively on the top level $\C v^{i,j}$ of the irreducible 
$\CW$-module $M^{i,j}$ \cite[Proposition 4.5]{DLY}. 
We can verify that there are no integers $i ,j$ which satisfy 
the above system of equations for any of  
$\lambda = \pm 6 k^2 (k + 2),  \pm 6 k^2 (3 k + 4)$.
This contradicts  \eqref{eq:eigenvalues}. 
Thus no irreducible $\CW$-module $M$ with weight $1$ top level 
can satisfy the exact sequence \eqref{eq:extention}.

The following theorem is proved.

\begin{thm}\label{thm:projectivity-CW}
$\CW$ is projective as a $\CW$-module.
\end{thm}

\section{$C_2$-cofiniteness of $K(\mathfrak{g},k)$: general case}\label{sec:C2_general_case}

In this section we show that the parafermion vertex operator algebra
$K(\mathfrak{g},k)$ is $C_2$-cofinite for any finite dimensional simple
Lie algebra $\mathfrak{g}$ and any positive integer $k$.
Our argument is a slightly modified version of that in 
\cite[Sections 3 and 4]{DW2}, which is based on a result of \cite{Li2}. 

For a vector space $X$, let $X^* = \Hom_\C(X,\C)$ be its dual space. 
A fact used in \cite{Li2} is that a vertex operator algebra $V$ is $C_2$-cofinite 
if and only if the dual space $(R_V)^*$ of $R_V = V/C_2(V)$ is contained in the restricted dual 
$V' = \bigoplus_n V_{(n)}^*$ of $V$. 
Indeed, the $\Z$-grading $R_V = \bigoplus_n (R_V)_n$ of \eqref{Z-grading-RV} 
implies that $(R_V)^* = \prod_n (R_V)_n^*$. Since $(R_V)_n^* \subset V_{(n)}^*$, 
we see that  $(R_V)^* \subset V'$ if and only if $(R_V)_n \ne 0$ 
for only finitely many $n$.

In \cite{Li2}, Li introduced a weak module $(D(M), Y^*)$ associated with 
a weak module $M$ for a vertex operator algebra $V$ \cite[Definition 2.4]{Li2}. 
Let $\la \,\cdot\,, \,\cdot\, \ra : M^* \times M \rightarrow \C$ be 
the natural pairing. For $v \in V$, $Y^*(v,z) = \sum_{n \in \Z} v^*_n z^{-n-1}$ 
is defined by 
\begin{equation*}
\la Y^*(v,z) w, u \ra = 
\la w, Y(e^{zL(1)} (-z^{-2})^{L(0)} v, z^{-1}) u \ra 
\end{equation*}
for $w \in M^*$ and $u \in M$. 
In fact, $(D(M), Y^*)$ is a unique maximal weak $V$-module contained in $(M^*, Y^*)$ 
\cite[Remark 2.9]{Li2}.  
An important property of $D(M)$ is \cite[Proposition 3.6]{Li2}
\begin{equation}\label{eq:MC2M_DM}
(M/C_2(M))^* \subset D(M),
\end{equation}
where $C_2(M)$ is the subspace of  $M$ spanned by the elements 
$v_{-2}u$, $v \in V$, $u \in M$ and $(M/C_2(M))^*$ is the set of 
$\eta \in M^*$ such that $\eta (C_2(M)) = 0$.

For a $V$-module $M = \bigoplus_{h \in \C} M_{(h)}$, a $V$-module 
structure $(M', Y')$ on its restricted dual $M' = \bigoplus_{h \in \C} M_{(h)}^*$ 
was defined and studied in \cite[Sections 5.2 and 5.3]{FHL}. The $V$-module $(M', Y')$ 
is called the contragredient module. The double cntragredient module 
$(M'', Y'')$ is naturally isomorphic to the original $V$-module $(M,Y)$ 
\cite[Proposition 5.3.1]{FHL}. 
In this case $M^* = \prod_{h \in \C} M_{(h)}^*$ 
and $Y^*(v,z)$ agrees with $Y'(v,z)$ on $M'$. 

We will consider $D(M)$ and $D(M')$ for a $V$-module $M = \bigoplus_{h \in \C} M_{(h)}$. 
A slightly different notation is used in \cite{DW2}. 
In fact,  $D(M')$ is denoted by $\mathcal{D}(M)$ in \cite{DW2}.
For the rest of this section, we follow the notation in \cite{DW2}. 
Note that the dual space $(M')^*$ of $M'$ is $\prod_h M_{(h)}$ 
since $M_{(h)}$ is finite dimensional. 
Thus
\begin{equation*}
\mathcal{D}(M) = \{ w \in \prod_h M_{(h)} \,|\, v_n w = 0 \text{ for } n \gg 0, v \in V\}.
\end{equation*}

For a vertex operator subalgebra $U$ of $V$, set
\begin{equation*}
\mathcal{D}_U(M) = \{ w \in \prod_h M_{(h)} \,|\, v_n w = 0 \text{ for } n \gg 0, v \in U\}.
\end{equation*}
Then $\mathcal{D}_U(M)$ is a unique maximal weak $U$-module contained in 
$\prod_h M_{(h)}$ and $\mathcal{D}(M) \subset \mathcal{D}_U(M)$.

\begin{lem}\label{lem:L0-semisimple}
Let $V = (V,Y,\1,\omega)$ be a vertex operator algebra and 
$U = (U,Y,\1,\omega_U)$ a vertex operator subalgebra with $\omega_U \in V_{(2)}$. 
Let $M = \bigoplus_{h \in \C} M_{(h)}$ be a $V$-module. 
Assume that $U$ is $C_2$-cofinite and that $L_U(0) = (\omega_U)_1$ is 
semisimple on $M$. 
Then $L_U(0)$ is semismple on $\mathcal{D}_U(M)$.
\end{lem}

\begin{proof}
Since $\mathcal{D}_U(M)$ is a weak $U$-module and since $U$ is $C_2$-cofinite, 
$\mathcal{D}_U(M)$ decomposes into a
direct sum of generalized eigenspaces for $L_U(0)$
\cite[Lemma 5.6, Proposition 5.7]{ABD}.
Let $w = (w^h)_h \in \mathcal{D}_U(M)$ be a generalized 
eigenvector for $L_U(0)$ with eigenvalue $\lambda$, where $w^h \in M_{(h)}$.
Then $(L_U(0) - \lambda)^m w = 0$ for some $m$.
Since $\omega_U \in V_{(2)}$, the operator $L_U(0)$ 
is a weight $0$ operator and it leaves $M_{(h)}$ invariant.
Then for each $h$, we have $(L_U(0) - \lambda)^m w^h = 0$.
Furthermore,  $L_U(0)$ is semisimple on $M_{(h)}$, 
since we are assuming that $L_U(0)$ is semisimple on $M$. 
Thus $L_U(0) w^h = \lambda w^h$ for all $h$ and so $L_U(0) w = \lambda w$.
Hence $\mathcal{D}_U(M)$ is a direct sum of eigenspaces for
$L_U(0)$ as desired.
\end{proof}

Since $\mathcal{D}(M)$ is an $L_U(0)$-invariant subspace of $\mathcal{D}_U(M)$, 
the above lemma implies that $L_U(0)$ is semisimple on  $\mathcal{D}(M)$ 
also. 
Taking $M$ to be $V$, we have the following corollary.

\begin{cor}\label{cor:L0-semisimple}
Let $V = (V,Y,\1,\omega)$ be a vertex operator algebra and 
$U = (U,Y,\1,\omega_U)$ a vertex operator subalgebra with $\omega_U \in V_{(2)}$. 
Assume that $U$ is $C_2$-cofinite and that $L_U(0) = (\omega_U)_1$ is 
semisimple on $V$. 
Then $L_U(0)$ is semismple on $\mathcal{D}(V)$.
\end{cor}

We need another lemma, which is a variant of \cite[Lemma 3.2]{DW2}.

\begin{lem}\label{lem:Vir-unitary}
Let $V = (V,Y,\1,\omega)$ be a simple vertex operator algebra not necessarily
$V_{(n)} = 0$ for $n < 0$ nor $V_{(0)} = \C\1$. 
Suppose there is a positive definite hermitian form $\la \,\cdot\,, \,\cdot\, \ra$
on $V$ such that $(V,\, \la \,\cdot\,, \,\cdot\, \ra)$ is a unitary representation 
of the Virasoro algebra $Vir$ generated by the component operators of
$Y(\omega,z) = \sum_{n \in \Z} L(n) z^{-n-2}$. 
Then  $V = \bigoplus_{n \ge 0} V_{(n)}$ with $V_{(0)} = \C\1$ and $L(1)V_{(1)} = 0$.
\end{lem}

\begin{proof}
By our hypothesis, $V$ is completely reducible as a $Vir$-module
and so $L(0)$ is semisimple on $V$. 
Moreover, the eigenvalues for $L(0)$ on $V$ are nonnegative. 
Thus $V_{(n)} = 0$ for $n < 0$. 

Let $0 \ne v \in V_{(0)}$. Then $v$ is a highest weight vector for $Vir$ with 
highest weight $0$. The $Vir$-module generated by $v$ is an irreducible 
highest weight module with highest weight $0$ and so $L(-1)v = 0$. 
Hence $v$ is a vacuum-like vector by \cite[Corollary 4.7.6]{LL}, 
\cite[Proposition 3.3]{Li}, that is, 
$u_n v = 0$ for all $n \ge 0$ and $u \in V$.  
Then $v \in \C\1$ by \cite[Proposition 3.11.4]{LL}, for $V$ is a simple 
vertex operator algebra. 
Thus  $V_{(0)} = \C\1$. 

For $v \in V_{(1)}$, we have $\la L(1)v, \1 \ra = \la v, L(-1)\1 \ra = 0$ 
by the unitarity. 
Since $L(1)V_{(1)} \subset V_{(0)} = \C\1$ and since $\la \,\cdot\,, \,\cdot\, \ra$ 
is positive definite, we conclude that $L(1)V_{(1)} = 0$.
\end{proof}

The following theorem is a modified version of \cite[Theorem 3.3]{DW2}, namely, we 
remove the rationality of vertex operator subalgebras from the assumption.

\begin{thm}\label{thm:C2_general}
Let $V = (V,Y,\1,\omega)$ be a simple vertex operator algebra not necessarily
$V_{(n)} = 0$ for $n < 0$ nor $V_{(0)} = \C\1$.
Assume that there exist finitely many vertex operator subalgebras
$V^i = (V^i,Y,\1,\omega^i)$ with $\omega^i \in V_{(2)}$, $1 \le i \le p$
of $V$ which satisfy the following three conditions: 
$($a$)$ Each $V^i$ is $C_2$-cofinite,
$($b$)$ $\omega = \sum_{i=1}^p a_i \omega^i$ for some positive constants $a_i$,
$($c$)$ There is a positive definite hermitian form $\la \,\cdot\,, \,\cdot\, \ra$
on $V$ such that $(V,\, \la \,\cdot\,, \,\cdot\, \ra)$ is a unitary representation 
of the Virasoro algebra $Vir^i$ generated by the component operators of
$Y(\omega^i,z) = \sum_{n \in \Z} L^i(n) z^{-n-2}$ for $1 \le i \le p$.
Then the the following assertions hold.

$(1)$ $V = \bigoplus_{n \ge 0} V_{(n)}$ with $V_{(0)} = \C\1$ and $L(1)V_{(1)} = 0$.

$(2)$ $V$ is $C_2$-cofinite.
\end{thm}

\begin{proof}
The conditions (b) and (c) imply that $(V,\, \la \,\cdot\,, \,\cdot\, \ra)$ is a unitary 
representation of the Virasoro algebra $Vir$ generated by the component 
operators of $Y(\omega,z) = \sum_{n \in \Z} L(n) z^{-n-2}$. 
Hence the assertion (1) is a consequence of Lemma \ref{lem:Vir-unitary}.

We will show that $\mathcal{D}(V) = V$. 
By the condition (c), $V$ is completely reducible as a $Vir^i$-module 
and so $L^i(0)$ acts on $V$ semisimply. 
Moreover, the eigenvalues for $L^i(0)$ on $V$ are nonnegative real numbers. 
It follows from Corollary \ref{cor:L0-semisimple} that $L^i(0)$ is semisimple on 
$\mathcal{D}(V)$ with nonnegative eigenvalues.

Now, we follow the proof of \cite[Theorem 3.3]{DW2}. 
Suppose $\mathcal{D}(V) \ne V$ and take $v = (v^n)_{n \ge 0} \in \mathcal{D}(V)$ 
not contained in $V$, where $v^n \in V_{(n)}$. 
Since $L^i(0)$ is semisimple on $\mathcal{D}(V)$, we write 
$v = \sum_{r = 1}^{j_i} v^{(i,r)}$, where $v^{(i,r)}$ is an eigenvector for $L^i(0)$ 
with eigenvalue $\lambda_r^i \in \R_{\ge 0}$ and  
$\lambda_r^i \ne \lambda_s^i$ if $r \ne s$. 
Let $\lambda^i$ be the maximum of $\lambda_r^i$, $1 \le r \le j_i$. 
Denote $v^{(i,r)}$ by $v^{(i,r)} = (v^{(i,r)n})_{n \ge 0}$ as an element of 
$\mathcal{D}(V)$ with $v^{(i,r)n} \in V_{(n)}$. 
Then $v^n =  \sum_{r = 1}^{j_i} v^{(i,r)n}$. 
Since $\omega^i \in V_{(2)}$, the operator $L^i(0) = \omega^i_1$ 
is a weight $0$ operator and it preserves $V_{(n)}$. 
Hence $L^i(0) v^{(i,r)n} = \lambda_r^i v^{(i,r)n}$. 
Moreover, $\la L^i(0) u, w \ra = \la u, L^i(0) w \ra$ for $u, w \in V$. 
Thus $\la v^{(i,r)n}, v^{(i,s)n} \ra = 0$ if $r \ne s$. 
Since $\la \,\cdot\,, \,\cdot\, \ra$ is positive definite, we have 
$\la L^i(0) v^n, v^n \ra \le \lambda^i \la  v^n, v^n \ra$. 
Recall that $L(0)v^n = n v^n$ and $L(0) = \sum_{i=1}^p a_i L^i(0)$. 
Hence $n \la v^n, v^n \ra \le \sum_{i=1}^p a_i  \lambda^i \la  v^n, v^n \ra$. 
Since $\sum_{i=1}^p a_i  \lambda^i$ is independent of $n$ and 
$\la \,\cdot\,, \,\cdot\, \ra$ is positive definite, this inequality implies that 
$v^n \ne 0$ for only finitely many $n$. But then $v \in V$, which is a 
contradiction. Therefore,  $\mathcal{D}(V) = V$. 

The assertion (1) implies that there is, up to scalar multiple, 
a unique nondegenerate symmetric invariant bilinear form on $V$ \cite{Li}. 
Then the $V$-module $(V,Y)$ is isomorphic to the contragredient module $(V',Y')$ 
\cite[Remark 5.3.3]{FHL}. 
In particular, $\mathcal{D}(V) = V$ implies $\mathcal{D}(V') = V'$. 
Then $(R_V)^* \subset V'$ by \eqref{eq:MC2M_DM}, 
since $D(V) = \mathcal{D}(V')$. 
Thus the assertion (2) holds.
\end{proof}

We use the notation of \cite{DW, DW2} for general parafermion vertex 
operator algebras. 
Thus $k$ is a positive integer, $\mathfrak{g}$ is a finite dimensional 
simple Lie algebra of rank $\ell$ with Cartan subalgebra $\h$, 
$V_{\widehat{\mathfrak{g}}}(k,0)$ is the vacuum Weyl module
for the affine Kac-Moody Lie algebra $\widehat{\mathfrak{g}}$ with level $k$ and
$L_{\widehat{\mathfrak{g}}}(k,0)$ is the simple quotient of the vacuum Weyl module. 
The parafermion vertex operator algebra $K(\mathfrak{g},k)$ is the commutant 
of the Heisenberg vertex operator algebra $M_{\widehat{\h}}(k,0)$ in 
$L_{\widehat{\mathfrak{g}}}(k,0)$. 
Let $\omega_{\mraff}$ and $\omega_{\h}$ be the conformal vectors of 
$L_{\widehat{\mathfrak{g}}}(k,0)$ and $M_{\widehat{\h}}(k,0)$, respectively. 
Then the conformal vector of $K(\mathfrak{g},k)$ is 
$\omega = \omega_{\mraff} - \omega_{\h}$.

Let $\Delta$ be the root system, 
$\Delta_+$ the set of positive roots, 
and $\la \,\cdot\,, \,\cdot\, \ra$ an invariant 
symmetric nondegenerate bilinear form on $\mathfrak{g}$ such that 
$\la \alpha,\alpha \ra = 2$ for a long root $\alpha$. 
For $\alpha \in \Delta_+$, 
let $\mathfrak{g}_\alpha$, $x_{\pm \alpha}$ and 
$h_\alpha$ be as in \cite{DW2}. 
Then 
$[x_{\alpha}, x_{-\alpha}] = h_{\alpha}$, 
$[h_{\alpha}, x_{\pm \alpha}] = \pm 2 x_{\pm \alpha}$ 
and $\mathfrak{g}^\alpha = \C x_{\alpha} + \C h_{\alpha} + \C x_{-\alpha}$ 
is a subalgebra of $\mathfrak{g}$ isomorphic to $sl_2$. 
For $\alpha \in \Delta$, set $k_{\alpha} = \frac{2}{\la \alpha,\alpha \ra} k$ and 
let $\omega_{\alpha}$ and $W^3_{\alpha}$ be as in \cite{DW2}. 
Then $k_{\alpha} \in \{ k, 2k, 3k \}$. 
Let $P_{\alpha}$ be the vertex operator subalgebra of $K(\mathfrak{g},k)$ 
generated by $\omega_{\alpha}$ and $W^3_{\alpha}$. 
By \cite[Propositions 4.5, 4.6]{DW}, $P_{\alpha}$ is isomorphic to the parafermion 
vertex operator algebra $K(\mathfrak{g}^\alpha, k_{\alpha})$ associated with 
$\mathfrak{g}^\alpha \cong sl_2$ and $\omega_\alpha$ is its conformal vector.

We apply Theorem \ref{thm:C2_general} to $V = K(\mathfrak{g},k)$ and its 
vertex operator subalgebras $V^\alpha = P_\alpha$, $\alpha \in \Delta_+$. 
Indeed,  $K(sl_2, 1) = \C$ and for $2 \le k_\alpha \le 4$, $K(sl_2, k_{\alpha})$ 
is isomorphic to a well-known rational and $C_2$-cofinite vertex operator algebra. 
In the case $k_\alpha \ge 5$, 
we know $K(sl_2, k_{\alpha})$ is $C_2$-cofinite by Theorem \ref{thm:embedding_RW_in_RL}. 
Thus $P_\alpha$ is $C_2$-cofinite for any positive integer $k$, 
and so the condition (a) of Theorem \ref{thm:C2_general} is satisfied. 
Moreover, the condition (b) is satisfied, since
\begin{equation*}
\omega = \sum_{\alpha \in \Delta_+} 
\frac{k(k_\alpha + 2)}{k_\alpha (k + h^\vee)} \omega_\alpha,
\end{equation*}
where $h^\vee$ is the dual Coxeter number of $\mathfrak{g}$ \cite[Section 2]{DW2}.  
Finally, the condition (c) is satisfied by \cite[Lemma 4.3]{DW2}. 
Therefore, the following theorem is a consequence of Theorem \ref{thm:C2_general}.

\begin{thm}\label{thm:C2_cofinte_general_case}
The parafermion vertex operator algebra  $K(\mathfrak{g},k)$ is $C_2$-cofinite 
for any finite dimensional simple Lie algebra $\mathfrak{g}$ and any positive integer $k$. 
\end{thm}

\section*{Acknowledgments}
The authors would like to thank Toshiyuki Abe, Toshiro Kuwabara, Takeshi Suzuki 
and Hiroshi Yamauchi for valuable discussions. 
Part of computation was done by a computer algebra system Risa/Asir.
The authors are indebted to Kazuhiro Yokoyama for helpful advice concerning
the system. 
Some part of the work was done while C. L. and H. Y.  were staying
at Kavli Institute for Theoretical Physics China, Beijing in July and August, 2000, 
T. A, C. L. and H. Y. were staying at National Center for Theoretical Sciences (South), 
Tainan in September, 2000, 
and T. A. and H. Y. were staying at Academia Sinica, Taipei in December, 2011. 
They are grateful to those institutes. 
Part of the results was announced at a conference in Varna, Bulgaria, June 2011.
H. Y. thanks Institute for Nuclear Research and Nuclear Energy, 
Bulgarian Academy of Science for kind hospitality.

T. A. was partially supported by JSPS Grant-in-Aid for Scientific
Research (B) No. 20340007 and JSPS Grant-in-Aid for Challenging Exploratory
Research No. 23654006. 
C. L. was partially supported by NSC grant
100-2628-M-001005-MY4 and National Center for Theoretical Sciences, Taiwan.
H. Y. was partially supported by
JSPS Grant-in-Aid for Scientific Research (C) No. 23540009.

\end{document}